\documentclass{amsart}

\usepackage{graphicx}
\usepackage[centertags]{amsmath}
\usepackage{amsfonts}
\usepackage{amssymb}
\usepackage{amsthm}
\usepackage{newlfont}
\usepackage{hyperref}

\newtheorem{thm}{Theorem}[section]
\newtheorem{cor}[thm]{Corollary}
\newtheorem{lem}[thm]{Lemma}
\newtheorem{prop}[thm]{Proposition}

\theoremstyle{definition}
\newtheorem{defn}[thm]{Definition}
\newtheorem{prob}[thm]{Problem}
\theoremstyle{remark}

\numberwithin{equation}{section}

\newcommand{\N}{\mathbb{N}}
\newcommand{\Z}{\mathbb{Z}}
\newcommand{\R}{\mathbb{R}}
\newcommand{\V}{\mathrm{V}}
\newcommand{\E}{\mathrm{E}}
\newcommand{\RP}{\mathrm{RPast}}

\newcommand{\acts}{\curvearrowright}

\newcommand{\Sym}{\mathrm{Sym}}

\newcommand{\sH}{\mathrm{H}}
\newcommand{\f}{f}
\newcommand{\vf}{\hat{f}}
\newcommand{\cex}{\mathcal{E}}

\newcommand{\Int}{\mathrm{Int}}
\newcommand{\Ext}{\mathrm{Ext}}
\renewcommand{\:}{\,:\,}

\begin{document}

\title{A subgroup formula for f-invariant entropy}
\author{Brandon Seward}
\address{Department of Mathematics, University of Michigan, 530 Church Street, Ann Arbor, MI 48109, U.S.A.}
\email{b.m.seward@gmail.com}
%\subjclass{37A35} %Entropy and other invariants, isomorphism, classification
%\thanks{This research was supported by a National Science Foundation Graduate Research Fellowship}

\begin{abstract}
We study a measure entropy for finitely generated free group actions called f-invariant entropy. The f-invariant entropy was developed by Lewis Bowen and is essentially a special case of his measure entropy theory for actions of sofic groups. In this paper we relate the f-invariant entropy of a finitely generated free group action to the f-invariant entropy of the restricted action of a subgroup. We show that the ratio of these entropies equals the index of the subgroup. This generalizes a well known formula for the Kolmogorov--Sinai entropy of amenable group actions. We then extend the definition of f-invariant entropy to actions of finitely generated virtually free groups. We also obtain a numerical virtual measure conjugacy invariant for actions of finitely generated virtually free groups.
\end{abstract}
\maketitle

\section{Introduction}

Recently Lewis Bowen \cite{B10a} defined a numerical measure conjugacy invariant for actions of finitely generated free groups, called f-invariant entropy. The f-invariant entropy is relatively easy to calculate, has strong similarities with the classical Kolmogorov--Sinai entropy of actions of amenable groups, and in fact agrees with the classical Kolmogorov--Sinai entropy when the finitely generated free group is just $\Z$. Moreover, f-invariant entropy is essentially a special, simpler case of the recently emerging entropy theory of sofic group actions being developed by Bowen (\cite{B10b}, \cite{B10c}, \cite{Ba}), Kerr--Li (\cite{KL}, \cite{KL11a}, \cite{KL11b}), Kerr (\cite{Ke}), and others (\cite{C}, \cite{Z}, \cite{ZC}). The classical Kolmogorov--Sinai entropy has unquestionably been a fundamental and powerful tool in the study of actions of amenable groups, and f-invariant entropy seems posed to take a similar role in the study of actions of finitely generated free groups. Bowen has already used f-invariant entropy to classify most Bernoulli shifts over finitely generated free groups up to measure conjugacy \cite{B10a}, and the classical Abramov--Rohlin and (under a few assumptions) Juzvinskii's addition formulas have been extended to actions of finitely generated free groups by Bowen \cite{B10d} and Bowen--Gutman \cite{BG}, respectively. However the theory surrounding f-invariant entropy is still quite young. The f-invariant entropy has been computed for a few specific examples and for a few special types of actions, but there has yet to emerge a thorough understanding of the behavior of f-invariant entropy in general. Furthermore, there is a significant lack of intuition relating to f-invariant entropy. In some cases f-invariant entropy behaves just like Kolmogorov--Sinai entropy, but in other cases it behaves in ways that are completely unprecedented. There is therefore a significant need to develop and understand the theory of f-invariant entropy. This paper serves as a piece of this large program. We focus here on the specific question as to what relationship there is, if any, between the f-invariant entropy of a group action and the f-invariant entropy of the restricted action of a subgroup.

Before stating the main theorem, we give a brief definition of f-invariant entropy. A more detailed treatment of the definition will be given in Section \ref{SEC FINV}. Let $G$ be a finitely generated free group, let $S$ be a free generating set for $G$, and let $G$ act on a standard probability space $(X, \mu)$ by measure preserving bijections. If $\alpha$ is a measurable partition of $X$ and $F \subseteq G$ is finite, then we define
$$F \cdot \alpha = \bigvee_{f \in F} f \cdot \alpha.$$
Recall that the \emph{Shannon entropy} of a countable measurable partition $\alpha$ of $X$ is
$$\sH(\alpha) = \sum_{A \in \alpha} -\mu(A) \cdot \log(\mu(A)).$$
Also recall that $\alpha$ is \emph{generating} if the smallest $G$-invariant $\sigma$-algebra containing $\alpha$ contains all measurable sets up to sets of measure zero. If there exists a generating partition $\alpha$ having finite Shannon entropy, then the f-invariant entropy of this action is defined to be
$$\f_G(X, \mu) = \lim_{n \rightarrow \infty} (1 - 2 r) \cdot \sH(B(n) \cdot \alpha) + \sum_{s \in S} \sH(s B(n) \cdot \alpha \vee B(n) \cdot \alpha),$$
where $r = |S|$ is the rank of $G$ and $B(n)$ is the ball of radius $n$ centered on $1_G$ with respect to the generating set $S$. Surprisingly, Bowen proved in \cite{B10a} and \cite{B10c} that the value $f_G(X, \mu)$ neither depends on the choice of free generating set $S$ nor on the choice of finite Shannon entropy generating partition $\alpha$. If there is no finite Shannon entropy generating partition for this action, then the f-invariant entropy is undefined.

Our main theorem is the following.

\begin{thm} \label{THM INTROMAIN}
Let $G$ be a finitely generated free group, and let $H \leq G$ be a subgroup of finite index. Let $G$ act on a standard probability space $(X, \mu)$ by measure preserving bijections, and let $H$ act on $(X, \mu)$ by restricting the action of $G$. Assume that the f-invariant entropy is defined for either the $G$ action or the $H$ action. Then the f-invariant entropy is defined for both actions and
$$\f_H(X, \mu) = |G : H| \cdot \f_G(X, \mu).$$
\end{thm}

We mention that the above theorem is a generalization of a well known property of Kolmogorov--Sinai entropy. Specifically, if $G$ is a countable amenable group, $H \leq G$ is a subgroup of finite index, and $G$ acts measure preservingly on a standard probability space $(X, \mu)$ then $h_H(X, \mu) = |G : H| \cdot h_G(X, \mu)$, where $h_H$ and $h_G$ denote the Kolmogorov--Sinai entropies of the $H$ and $G$ actions, respectively (see \cite[Theorem 2.16]{D01} for a stronger result).

We give an example to show that $\f_H(X, \mu)$ may not equal $|G : H| \cdot \f_G(X, \mu)$ if $H$ is not of finite index, even if both $\f_H(X, \mu)$ and $\f_G(X, \mu)$ are defined. This is in contrast with Kolmogorov--Sinai entropy where $h_H(X, \mu) = |G : H| \cdot h_G(X, \mu)$, regardless if $H$ has finite index or infinite index. If $H$ has infinite index in $G$ then we take this equation to mean that $h_G(X, \mu) = 0$ if $h_H(X, \mu)$ is finite, and $h_H(X, \mu) = \infty$ if $h_G(X, \mu)$ is non-zero. We apply a similar logic to the equation $\f_H(X, \mu) = |G : H| \cdot \f_G(X, \mu)$ when $|G : H| = \infty$. We do however obtain the following relationship.

\begin{cor}
Let $G$ be a finitely generated free group, let $H \leq G$ be a non-trivial subgroup of infinite index, let $G$ act on a standard probability space $(X, \mu)$ by measure preserving bijections, and let $H$ act on $(X, \mu)$ by restricting the action of $G$. Suppose there are infinitely many finite index subgroups of $G$ containing $H$. If $f_H(X, \mu)$ is defined, then $f_G(X, \mu)$ is defined and $f_G(X, \mu) \leq 0$.
\end{cor}

The proof of the main theorem relies primarily on a study of Markov processes over free groups. In fact we first obtain Theorem \ref{THM INTROMAIN} for Markov processes and normal subgroups. We then use various arguments to extend the result to general subgroups and general actions. The definition of Markov processes is somewhat technical, so we postpone it until Section \ref{SEC MARKOV}.

The following result on Markov processes is a key ingredient for our arguments and also seems to be of general interest. 

\begin{thm}
Let $G$ be a free group, and let $H \leq G$ be a subgroup of finite index. Let $G$ act on a standard probability space $(X, \mu)$ by measure preserving bijections, and let $H$ act on $(X, \mu)$ by restricting the action of $G$. If $G \acts (X, \mu)$ is measurably conjugate to a Markov process then $H \acts (X, \mu)$ is measurably conjugate to a Markov process as well.
\end{thm}

We show that in many circumstances the property of being a Markov process is independent of the choice of a free generating set for $G$.

\begin{cor}
Let $G$ be a finitely generated free group acting measure preservingly on a standard probability space $(X, \mu)$. Let $S_1$ and $S_2$ be two free generating sets for $G$. Suppose that $G \acts (X, \mu)$ is measurably conjugate to a $S_1$-Markov process with finite Shannon entropy Markov partition. Then $G \acts (X, \mu)$ is measurably conjugate to a $S_2$-Markov process as well.
\end{cor}

Our main theorem also leads to the following interesting inequality involving f-invariant entropy. Relevant definitions can be found in the next section.

\begin{cor}
Let $G$ be a finitely generated free group acting on a standard probability space $(X, \mu)$ by measure preserving bijections. Suppose that this action admits a generating partition $\alpha$ having finite Shannon entropy. Then for any free generating set $S$ for $G$ and any finite right $S$-connected set $\Delta \subseteq G$ we have
$$f_G(X, \mu) \leq \frac{\sH(\Delta \cdot \alpha)}{|\Delta|} \leq \sH(\alpha).$$
\end{cor}

The rest of our corollaries deal with virtually free groups and the virtual measure conjugacy relation. Recall that a group $\Gamma$ is \emph{virtually free} if it contains a free subgroup of finite index. Similarly, a group is \emph{virtually $\Z$} if it contains $\Z$ as a finite index subgroup.

\begin{cor}
Let $\Gamma$ be a finitely generated virtually free group acting measure preservingly on a standard probability space $(X, \mu)$. Let $G, H \leq \Gamma$ be finite index free subgroups, and let them act on $(X, \mu)$ by restricting the $\Gamma$ action. Assume that there is a finite Shannon entropy generating partition for $\Gamma \acts (X, \mu)$. Then $f_G(X, \mu)$ and $f_H(X, \mu)$ are defined and
$$\frac{1}{|\Gamma : G|} \cdot f_G(X, \mu) = \frac{1}{|\Gamma : H|} \cdot f_H(X, \mu).$$
Furthermore, if $\Gamma$ is itself free then the above common value is $f_\Gamma(X, \mu)$.
\end{cor}

This corollary allows us to extend the definition of f-invariant entropy to actions of finitely generated virtually free groups.

\begin{defn}
Let $\Gamma$ be a finitely generated virtually free group acting measure preservingly on a standard probability space $(X, \mu)$. If there is a generating partition for this action having finite Shannon entropy, then we define the \emph{f-invariant entropy of $\Gamma \acts (X, \mu)$} to be
$$\f_\Gamma(X, \mu) = \frac{1}{|\Gamma: G|} \cdot \f_G(X, \mu),$$
where $G \leq \Gamma$ is any free subgroup of finite index, and $G$ acts on $X$ be restricting the action of $\Gamma$. If there is no generating partition for $\Gamma \acts (X, \mu)$ having finite Shannon entropy, then the f-invariant entropy of this action is undefined.
\end{defn}

The quantity $f_\Gamma(X, \mu)$ is a measure conjugacy invariant, and by the previous corollary this value does not depend on the choice of free subgroup of finite index $G$.

Next we consider virtual measure conjugacy among actions of finitely generated virtually free groups. Recall that two measure preserving actions $G \acts (X, \mu)$ and $H \acts (Y, \nu)$ on standard probability spaces are \emph{virtually measurably conjugate} if there are subgroups of finite index $G' \leq G$ and $H' \leq H$ such that the restricted actions $G' \acts (X, \mu)$ and $H' \acts (Y, \nu)$ are measurably conjugate, meaning that there is a group isomorphism $\psi: G' \rightarrow H'$ and a measure space isomorphism $\phi: X \rightarrow Y$ such that $\phi(g' \cdot x) = \psi(g') \cdot \phi(x)$ for every $g' \in G'$ and $\mu$-almost every $x \in X$.

\begin{cor}
For $i = 1, 2$, let $\Gamma_i$ be a finitely generated virtually free group which is not virtually $\Z$, and let $\Gamma_i$ act measure preservingly on a standard probability space $(X_i, \mu_i)$. Let $G_i \leq \Gamma_i$ be a free subgroup of finite index, and let $G_i$ act on $(X_i, \mu_i)$ by restricting the $\Gamma_i$ action. Assume that for each $i$ there is a finite Shannon entropy generating partition for $\Gamma_i \acts (X_i, \mu_i)$. If $\Gamma_1 \acts (X_1, \mu_1)$ is virtually measurably conjugate to $\Gamma_2 \acts (X_2, \mu_2)$ then
$$\frac{1}{r(G_1) - 1} \cdot f_{G_1}(X_1, \mu_1) = \frac{1}{r(G_2) - 1} \cdot f_{G_2}(X_2, \mu_2),$$
where $r(G_i)$ is the rank of $G_i$.
\end{cor}

This corollary allows us to define a numerical invariant for virtual measure conjugacy among actions of finitely generated virtually free groups which are not virtually $\Z$.

\begin{defn}
Let $\Gamma$ be a finitely generated virtually free group which is not virtually $\Z$, and let $\Gamma$ act measure preservingly on a standard probability space $(X, \mu)$. If there is a generating partition having finite Shannon entropy, then the \emph{virtual f-invariant entropy of $\Gamma \acts (X, \mu)$} is defined as
$$\vf_\Gamma(X, \mu) = \frac{1}{r(G) - 1} \cdot f_G(X, \mu),$$
where $G$ is any free subgroup of finite index, $r(G)$ is the rank of $G$, and $G$ acts on $(X, \mu)$ by restricting the $\Gamma$ action. If there is no generating partition with finite Shannon entropy, then the virtual f-invariant entropy of this action is undefined.
\end{defn}

The previous corollary shows that the quantity $\vf_\Gamma(X, \mu)$ is a virtual measure conjugacy invariant and does not depend on the choice of free subgroup of finite index $G$.

We furthermore show that virtual f-invariant entropy is a complete virtual measure conjugacy invariant for those Bernoulli shifts on which it is defined.

\begin{prop}
For $i = 1, 2$, let $(K_i^{\Gamma_i}, \mu_i^{\Gamma_i})$ be a Bernoulli shift over a finitely generated virtually free group $\Gamma_i$ with $\Gamma_i$ not virtually $\Z$. If the virtual f-invariant entropy $\vf_{\Gamma_i}(K_i^{\Gamma_i}, \mu_i^{\Gamma_i})$ is defined for each $i$, then $(K_1^{\Gamma_1}, \mu_1^{\Gamma_1})$ is virtually measurably conjugate to $(K_2^{\Gamma_2}, \mu_2^{\Gamma_2})$ if and only if $\vf_{\Gamma_1}(K_1^{\Gamma_1}, \mu_1^{\Gamma_1}) = \vf_{\Gamma_2}(K_2^{\Gamma_2}, \mu_2^{\Gamma_2})$.
\end{prop}

\subsection*{Organization}
In Section \ref{SEC DEFN} we cover basic definitions and notations. Then in Section \ref{SEC FINV} we define and discuss f-invariant entropy in detail. We discuss Markov processes in Section \ref{SEC MARKOV} and establish some of their basic properties. In Section \ref{SEC SUBFORM} we prove the main theorem and deduce some of its corollaries. Finally in Section \ref{SEC VIRT} we discuss applications to virtually free groups and virtual measure conjugacy.

\subsection*{Acknowledgments}
This material is based upon work supported by the National Science Foundation Graduate Student Research Fellowship under Grant No. DGE 0718128. The author would like to thank his advisor, Ralf Spatzier, for many helpful discussions and Hanfeng Li for comments on an earlier version of this paper.

\section{Definitions and Notation} \label{SEC DEFN}

In this paper all groups are assumed to be countable. We will work almost entirely with free groups, and thus there is an important distinction between multiplication on the left and multiplication on the right. We will have to work with both left-sided and right-sided notions simultaneously, and as we will point out later on, this seems to be absolutely necessary. We therefore will always use very careful notation and will always explicitly state whether we are working with multiplication on the right or with multiplication on the left.

Let $G$ be a finitely generated free group, and let $S$ be a free generating set for $G$. The \emph{rank of $G$} is the minimum size of a generating set for $G$, which in this case would be $|S|$. We denote the identity group element of $G$ by $1_G$. For $1_G \neq g \in G$, the \emph{reduced $S$-word representation of $g$} is the unique tuple $(s_1, s_2, \ldots, s_k)$ with the properties that $s_i \in S \cup S^{-1}$, $s_{i+1} \neq s_i^{-1}$, and $g = s_1 s_2 \cdots s_k$. The \emph{$S$-word length} of $g \in G$ is the length of the reduced $S$-word representation of $g$. The identity $1_G$ has $S$-word length $0$. The \emph{$S$-ball of radius $n$ in $G$ centered on $1_G$}, denoted $B_S(n)$, is the set of group elements whose $S$-word length is less than or equal to $n$. If $H \leq G$ is a subgroup, then the \emph{left $H$-cosets} are the sets $g H$ for $g \in G$. Similarly the \emph{right $H$-cosets} are the sets $H g$ for $g \in G$. A set $\Delta$ is a \emph{transversal} of the left (right) $H$-cosets if each left (right) $H$-coset meets $\Delta$ in precisely one point.

The \emph{right $S$-Cayley graph of $G$} is the graph with vertex set $G$ and edge set $\{(g, g s) \: g \in G, \ s \in S \cup S^{-1}\}$. Since $G$ is a free group and $S$ is a free generating set for $G$, the right $S$-Cayley graph of $G$ is a tree. When working with a graph $\Gamma$, we let $\V(\Gamma)$ and $\E(\Gamma)$ denote the vertex set and the edge set of $\Gamma$, respectively. A \emph{right $S$-path} is a non-self-intersecting path in the right $S$-Cayley graph of $G$. A set $F \subseteq G$ is \emph{right $S$-connected} if for every two elements $f_1, f_2 \in F$ the unique right $S$-path from $f_1$ to $f_2$ traverses only vertices in $F$. The \emph{right $S$-connected components} of $F \subseteq G$ are the maximal subsets of $F$ which are right $S$-connected. For three subsets $U, V, W \subseteq G$, we say that $V$ \emph{right $S$-separates $(U, W)$} if for every $u \in U$ and $w \in W$ the unique right $S$-path from $u$ to $w$ traverses some vertex in $V$. The \emph{right $S$-distance} between two elements $g, h \in G$ is defined to be the number of edges traversed by the unique right $S$-path from $g$ to $h$. We will also use the right $S$-distance implicitly when we refer to points which are right $S$-furthest from one another or right $S$-closest to one another. We say that $g, h \in G$ are \emph{right $S$-adjacent} if there is $s \in S \cup S^{-1}$ with $g s = h$. For $u, v \in G$, we define the \emph{right $S$-past of $u$ through $v$}, denoted $\RP_S(v, u)$, to be the set of $g \in G$ for which the unique right $S$-path from $g$ to $u$ traverses $v$. If $U, V \subseteq G$, then we define
$$\RP_S(V, U) = \bigcap_{u \in U} \bigcup_{v \in V} \RP_S(v, u).$$

The reader is encouraged to think carefully about the definition of $\RP_S(v, u)$. The truth is that the word ``past'' is somewhat misleading. As an example to consider, the right $S$-past of $1_G$ through $s \in S$, $\RP_S(s, 1_G)$, is the set of all group elements whose reduced $S$-word representations begin on the left with $s$. This can be misleading as some may be inclined to think of this set as the future. The generating set $S$ provides us with $2 |S|$ directions of movement, and we can consider any such direction as the past. Also notice that $V$ right $S$-separates $(U, W)$ if and only if $W \subseteq \RP_S(V, U)$ if and only if $U \subseteq \RP_S(V, W)$.

The \emph{left $S$-Cayley graph of $G$}, the \emph{left $S$-paths}, the \emph{left $S$-connected sets}, the \emph{left $S$-distance} between a pair of group elements, etc. are defined in a fashion similar to their right counterparts. We call a set $F \subseteq G$ \emph{bi-$S$-connected} if it is both right $S$-connected and left $S$-connected.

Unless stated otherwise, we will use the term \emph{group action} and the notation $G \acts (X, \mu)$ to mean a countable group $G$ acting on a standard probability space $(X, \mu)$ by measure preserving bijections. Our probability spaces will always be assumed to be standard probability spaces. Also, if $G$ acts on $(X, \mu)$ and $H \leq G$ is a subgroup, then we will always implicitly let $H$ act on $(X, \mu)$ by restricting the $G$ action. We will never consider any other types of actions of subgroups. Two actions $G \acts (X, \mu)$ and $G \acts (Y, \nu)$ are \emph{measurably conjugate} if there exists an isomorphism of measure spaces $\phi: (X, \mu) \rightarrow (Y, \nu)$ such that $\phi(g \cdot x) = g \cdot \phi(x)$ for every $g \in G$ and $\mu$-almost every $x \in X$. Similarly, if $G$ acts continuously on two topological spaces $X$ and $Y$, then $X$ and $Y$ are \emph{topologically conjugate} if there is a homeomorphism $\phi: X \rightarrow Y$ such that $\phi(g \cdot x) = g \cdot \phi(x)$ for every $g \in G$ and every $x \in X$.

Let $G$ act on $(X, \mu)$. If $\alpha$ and $\beta$ are measurable partitions of $X$, then $\beta$ is \emph{coarser} than $\alpha$, or $\alpha$ is a \emph{refinement} of $\beta$, if every member of $\beta$ is a union of members of $\alpha$. If $\beta$ is coarser than $\alpha$ then we write $\beta \leq \alpha$. For two partitions $\alpha = \{A_i \: i \in I\}$ and $\beta = \{B_j \: j \in J\}$ of $X$ we define their \emph{join} to be the partition
$$\alpha \vee \beta = \{A_i \cap B_j \: i \in I, \ j \in J\}.$$
We similarly define the join $\bigvee_{i = 1}^n \alpha_i$ of a finite number of partitions $\{\alpha_i \: 1 \leq i \leq n\}$. For a countably infinite collection of partitions $\{\alpha_i \: i \in I\}$ of $X$, we let
$$\bigvee_{i \in I} \alpha_i$$
denote the smallest $\sigma$-algebra containing all of the members of all of the $\alpha_i$'s. If $\{\mathcal{F}_i \: i \in I\}$ is a collection of $\sigma$-algebras on $X$, then we let $\bigvee_{i \in I} \mathcal{F}_i$ denote the smallest $\sigma$-algebra containing all of the sets of each of the $\mathcal{F}_i$'s. If $\alpha = \{A_i \: i \in I\}$ is a partition of $X$ then for $g \in G$ we define
$$g \cdot \alpha = \{g \cdot A_i \: i \in I\}.$$
Similarly, for $F \subseteq G$ we define
$$F \cdot \alpha = \bigvee_{f \in F} f \cdot \alpha.$$
Notice that $F \cdot \alpha$ is a $\sigma$-algebra if $F$ is infinite and that $g \cdot \alpha = \{g\} \cdot \alpha$ for every $g \in G$. A measurable countable partition $\alpha$ is \emph{generating} if for every measurable set $B \subseteq X$ there is a set $B' \in G \cdot \alpha$ with $\mu(B \triangle B') = 0$. The \emph{Shannon entropy} of a countable measurable partition $\alpha$ is
$$\sH(\alpha) = \sum_{A \in \alpha} -\mu(A) \cdot \log(\mu(A)).$$
If $\beta$ is another countable measurable partition of $X$, then the \emph{conditional Shannon entropy} of $\alpha$ relative to $\beta$ is
$$\sH(\alpha / \beta) = \sum_{B \in \beta} \mu(B) \cdot \left( \sum_{A \in \alpha} -\frac{\mu(A \cap B)}{\mu(B)} \cdot \log \left( \frac{\mu(A \cap B)}{\mu(B)} \right) \right).$$
If $\mathcal{F}$ is a $\sigma$-algebra on $X$ consisting of measurable sets and $f : X \rightarrow \R$ is a measurable function, then we denote the \emph{conditional expectation of $f$ relative to $\mathcal{F}$} by $\cex(f / \mathcal{F})$. Recall that $\cex(f / \mathcal{F})$ is the unique $\mathcal{F}$-measurable function, up to agreement $\mu$-almost everywhere, with the property that for every $\mathcal{F}$-measurable function $h: X \rightarrow \R$
$$\int_X h \cdot f d \mu = \int_X h \cdot \cex(f / \mathcal{F}) d \mu.$$
If $\alpha$ is a countable measurable partition, then we define $\cex(f / \alpha) = \cex(f / \mathcal{F})$ where $\mathcal{F}$ is the $\sigma$-algebra generated by $\alpha$. We define the \emph{conditional Shannon entropy} of a countable measurable partition $\alpha$ relative to a sub-$\sigma$-algebra $\mathcal{F}$ by
$$\sH(\alpha / \mathcal{F}) = \int_X \sum_{A \in \alpha} - \cex(\chi_A / \mathcal{F}) \cdot \log(\cex(\chi_A / \mathcal{F})),$$
where $\chi_A$ is the characteristic function of $A$. It is well known that if $\beta$ is a countable measurable partition of $X$ and $\mathcal{F}$ is the $\sigma$-algebra generated by $\beta$, then $\sH(\alpha / \beta) = \sH(\alpha / \mathcal{F})$.

The following lemma lists some well known properties of Shannon entropy which we will need (see \cite{Gl03} for a proof).

\begin{lem} \label{LEM SHAN}
Let $(X, \mu)$ be a standard probability space, let $\alpha$ and $\beta$ be countable measurable partitions of $X$, and let $\mathcal{F}$, $\mathcal{F}'$, and $(\mathcal{F}_i)_{i \in \N}$ be $\sigma$-algebras on $X$ consisting of measurable sets. Assume that $\mathcal{F} \subseteq \mathcal{F}'$. Then
\begin{enumerate}
\item[\rm (i)] $\sH(\alpha \vee \beta) = \sH(\alpha / \beta) + \sH(\beta)$;
\item[\rm (ii)] $\sH(\alpha / \beta \vee \mathcal{F}) + \sH(\beta / \mathcal{F}) = \sH(\alpha \vee \beta / \mathcal{F}) = \sH(\beta / \alpha \vee \mathcal{F}) + \sH(\alpha / \mathcal{F});$
\item[\rm (iii)] $\sH(\alpha / \mathcal{F}') \leq \sH(\alpha / \mathcal{F})$;
\item[\rm (iv)] $\sH(\alpha / \bigvee_{i \in \N} \mathcal{F}_i) = \lim_{n \rightarrow \infty} \sH(\alpha / \bigvee_{i = 1}^n \mathcal{F}_i)$;
\end{enumerate}
Furthermore, if $\cex(\chi_A | \mathcal{F}')(x) = \cex(\chi_A | \mathcal{F})(x)$ for every $A \in \alpha$ and $\mu$-almost every $x \in X$ then equality holds in clause (iii). Conversely, if $\sH(\alpha) < \infty$ and equality holds in (iii), then $\cex(\chi_A | \mathcal{F}')(x) = \cex(\chi_A | \mathcal{F})(x)$ for every $A \in \alpha$ and $\mu$-almost every $x \in X$.
\end{lem}

\section{f-invariant entropy} \label{SEC FINV}

Let $G$ be a finitely generated free group, let $S$ be a free generating set for $G$, and let $G$ act on $(X, \mu)$. For a countable measurable partition $\alpha$ with $\sH(\alpha) < \infty$ we define
$$F_G(X, \mu, S, \alpha) = (1 - 2r) \cdot \sH(\alpha) + \sum_{s \in S} \sH(s \cdot \alpha \vee \alpha),$$
where $r = |S|$ is the rank of $G$. Notice that by clause (i) of Lemma \ref{LEM SHAN} we can rewrite this expression in two ways:
$$F_G(X, \mu, S, \alpha) = (1 - r) \cdot \sH(\alpha) + \sum_{s \in S} \sH(s \cdot \alpha / \alpha);$$
$$F_G(X, \mu, S, \alpha) = \sH(\alpha) + \sum_{s \in S} (\sH(s \cdot \alpha / \alpha) - \sH(\alpha)).$$
All three ways of expressing $F_G(X, \mu, S, \alpha)$ will be useful to us. We define the \emph{f-invariant entropy rate of $(S, \alpha)$} to be
$$f_G(X, \mu, S, \alpha) = \lim_{n \rightarrow \infty} F_G(X, \mu, S, B_S(n) \cdot \alpha),$$
where $B_S(n)$ is the ball of radius $n$ in $G$, with respect to the generating set $S$, centered on the identity. Regarding the existence of this limit, Bowen proved the following.

\begin{lem}[Bowen, \cite{B10a}] \label{LEM FDECR}
Let $G$ be a finitely generated free group, let $S$ be a free generating set for $G$, and let $G$ act on $(X, \mu)$. If $U \subseteq V \subseteq G$ are finite and every left $S$-connected component of $V$ meets $U$, then for every countable measurable partition $\alpha$ with $\sH(\alpha) < \infty$ we have
$$F_G(X, \mu, S, V \cdot \alpha) \leq F_G(X, \mu, S, U \cdot \alpha).$$
\end{lem}

In particular, the terms appearing in the limit defining $f_G(X, \mu, S, \alpha)$ are non-increasing and thus the limit exists, although it may be negative infinity. If there exists a generating partition $\alpha$ having finite Shannon entropy, then the \emph{f-invariant entropy of $G \acts (X, \mu)$} is defined to be
$$f_G(X, \mu) = f_G(X, \mu, S, \alpha).$$
If there is no generating partition having finite Shannon entropy, then the f-invariant entropy of the action is not defined. Amazingly, the value of the f-invariant entropy does not depend on the choice of generating partition nor on the choice of free generating set for $G$, as the following theorem of Bowen states.

\begin{thm}[Bowen, \cite{B10a}, \cite{B10c}]
Let $G$ be a finitely generated free group acting on a probability space $(X, \mu)$. If $S$ and $T$ are free generating sets for $G$ and $\alpha$ and $\beta$ are generating partitions with finite Shannon entropy, then
$$f_G(X, \mu, S, \alpha) = f_G(X, \mu, T, \beta).$$
\end{thm}

A simple computation shows that when $G = \Z$ the f-invariant entropy is identical to the classical Kolmogorov--Sinai entropy. Furthermore, in \cite{B10a} Bowen calculated the f-invariant entropy of a Bernoulli shift $(K^G, \mu^G)$ to be the same as in the setting of amenable groups:
$$f_G \left(K^G, \mu^G \right) = \sum_{k \in K} - \mu(k) \cdot \log(\mu(k))$$
under the assumption that the support of $\mu$ is countable and this sum is finite. If the support of $\mu$ is not countable or the sum above is not finite, then  the f-invariant entropy is undefined (as Kerr--Li \cite{KL11b} proved there can be no finite Shannon entropy generating partition). Bowen further proved that f-invariant entropy is a complete invariant for measure conjugacy among the Bernoulli shifts on which it is defined. This generalizes the famous theorems of Ornstein (\cite{O70a}) and Kolmogorov (\cite{Ko58}, \cite{Ko59}).

We remark that f-invariant entropy involves taking some sort of ``average'' over the balls $B_S(n)$, just as Kolmogorov--Sinai entropy involves averaging over F{\o}lner sets. Since balls in free groups have relatively large boundary, the ``averaging'' happens by letting the interior of the ball and the boundary of the ball nearly completely cancel one another, leaving an ``average'' value behind. This intuitive viewpoint is based on the fact that if $K \subseteq G$ is finite and left $S$-connected then
$$1 = (1 - 2r)|K| + \sum_{s \in S} |s K \cup K|,$$
as the reader is invited to verify by induction (compare this with $F_G(X, \mu, S, K \cdot \alpha)$).

While f-invariant entropy does share some strong similarities with Kolmogorov--Sinai entropy, it also possesses some properties which are somewhat baffling from the classical entropy theory perspective. For example, a short computation shows that if $G$ acts on a set of $n$ points equipped with the uniform probability measure then the f-invariant entropy of this action is $(1 - r) \cdot \log(n)$, where $r$ is the rank of $G$. If $n > 1$ and $G \neq \Z$ then this value is finite and negative! Another strange property is that the f-invariant entropy of a factor can be larger than the f-invariant entropy of the original action \cite{B10e}.

\section{Markov processes} \label{SEC MARKOV}

Markov processes are somewhat similar to Bernoulli shifts as they are characterized by the existence of a generating partition with strong independence properties. We point out that when f-invariant entropy is not involved, we discuss Markov processes in the context of free groups without any finite generation assumption. However we do assume that all of our free groups are countable.

\begin{defn}[Bowen, \cite{B10d}]
Let $G$ be a free group, let $S$ be a free generating set for $G$, let $G$ act on $(X, \mu)$, and let $\alpha$ be a countable measurable partition of $X$. We call $X$ a \emph{$(S, \alpha)$-Markov process} if $\alpha$ is a generating partition and for every $A \in \alpha$, $s \in S \cup S^{-1}$, and $\mu$-almost every $x \in X$
$$\cex \left(\chi_{s \cdot A} / \RP_S(1_G, s) \cdot \alpha \right)(x) = \cex(\chi_{s \cdot A} / \alpha)(x),$$
where $\chi_{s \cdot A}$ is the characteristic function of the set $s \cdot A$. We say that $X$ is a \emph{$\alpha$-Markov process} if it is a $(S, \alpha)$-Markov process for some $S$, and we similarly say that $X$ is a \emph{$S$-Markov process} if it is a $(S, \alpha)$-Markov process for some $\alpha$. If $X$ is a $\alpha$-Markov process, then we call $\alpha$ a \emph{Markov partition}. Finally, we say that $X$ is a \emph{Markov process} if it is a $(S, \alpha)$-Markov process for some $S$ and some $\alpha$.
\end{defn}

In the next section we will show that under a mild assumption the property of being a Markov process does not depend on the free generating set $S$ chosen for $G$ (the Markov partition however will depend on the free generating set chosen).

Our interest in Markov processes comes from the fact that the formulas for both Shannon entropy and f-invariant entropy simplify. The reason why this simplification occurs is due to Lemma \ref{LEM SHAN}. That lemma immediately leads to an alternate characterization of Markov processes which is substantially easier to work with.

\begin{lem}[Bowen, \cite{B10d}] \label{LEM DEFN}
Let $G$ be a free group, let $S$ be a free generating set for $G$, let $G$ act on $(X, \mu)$, and let $\alpha$ be a countable measurable partition of $X$ with $\sH(\alpha) < \infty$. Then $X$ is a $(S, \alpha)$-Markov process if and only if $\alpha$ is generating and
$$\sH \left(s \cdot \alpha / \RP_S(1_G, s) \cdot \alpha \right) = \sH(s \cdot \alpha / \alpha)$$
for every $s \in S \cup S^{-1}$.
\end{lem}

As a convenience to the reader, we include the proof below.

\begin{proof}
First suppose that $X$ is a $(S, \alpha)$-Markov process. Then $\alpha$ is a generating partition and it immediately follows from the definition of conditional Shannon entropy that
$$\sH(s \cdot \alpha / \RP_S(1_G, s) \cdot \alpha) = \sH(s \cdot \alpha / \alpha).$$
Now suppose that $\alpha$ is a generating partition and
$$\sH(s \cdot \alpha / \RP_S(1_G, s) \cdot \alpha) = \sH(s \cdot \alpha / \alpha)$$
for every $s \in S \cup S^{-1}$. As $\sH(\alpha) < \infty$, it immediately follows from Lemma \ref{LEM SHAN} that $X$ is a $(S, \alpha)$-Markov process.
\end{proof}

Thus various conditional Shannon entropies can simplify substantially when working with Markov processes. This fact is also evident in the next lemma.

\begin{defn} \label{DEFN REDGE}
Let $G$ be a free group, and let $S$ be a free generating set for $G$. If $F \subseteq G$ is finite and right $S$-connected, then we define an element $R_S(F)$ in the additive abelian group $\bigoplus_{s \in S} \Z \cdot s$ by setting
$$R_S(F) = \sum_{s \in S} a_s \cdot s,$$
where $a_s$ is the number of pairs $(g, gs)$ with $g, gs \in F$.
\end{defn}

\begin{lem} \label{LEM SHANEDGE}
Let $G$ be a free group acting on a probability space $(X, \mu)$. Suppose that $X$ is a $(S, \alpha)$-Markov process with $\sH(\alpha) < \infty$. If $F \subseteq G$ is finite and right $S$-connected then
$$\sH(F \cdot \alpha) = \sH(\alpha) + \zeta(R_S(F)),$$
where $\zeta : (\bigoplus_{s \in S} \Z \cdot s) \rightarrow \R$ is the linear extension of the map $s \mapsto \sH(s \cdot \alpha / \alpha)$.
\end{lem}

\begin{proof}
We first point out that by clause (i) of Lemma \ref{LEM SHAN}
$$\sH(s^{-1} \cdot \alpha / \alpha) = \sH(s^{-1} \cdot \alpha \vee \alpha) - \sH(\alpha) = \sH(\alpha \vee s \cdot \alpha) - \sH(\alpha) = \sH(s \cdot \alpha / \alpha),$$
where the second equality is due to the action of $G$ being measure preserving.

Now we proceed to prove the lemma. We use induction on the cardinality of $F$. If $|F| = 1$ and $F = \{f\}$, then $R_S(F) = 0$ and since $G \acts (X, \mu)$ is measure preserving we have
$$\sH(F \cdot \alpha) = \sH(f \cdot \alpha) = \sH(\alpha) = \sH(\alpha) + \zeta(R_S(F)).$$
Now suppose this property holds whenever $|F| \leq q$. Let $F$ be a finite right $S$-connected set with $|F| = q + 1$. Let $f \in F$ be an element with maximum $S$-word length, and set $F' = F \setminus \{f\}$. Then $F'$ is right $S$-connected. Let $t \in S \cup S^{-1}$ be such that $f \in F' t$. Set $f_0 = f t^{-1}$. Then by our choice of $f$ we have that $F' \subseteq \RP_S(f_0, f)$ and hence $f_0^{-1} F' \subseteq \RP_S(1_G, f_0^{-1} f) = \RP_S(1_G, t)$. So it follows from Lemma \ref{LEM DEFN} and clause (iii) of Lemma \ref{LEM SHAN} that
$$\sH(t \cdot \alpha / \alpha) = \sH(t \cdot \alpha / \RP_S(1_G, t) \cdot \alpha)$$
$$\leq \sH(t \cdot \alpha / f_0^{-1} F' \cdot \alpha) \leq \sH(t \cdot \alpha / \alpha).$$
Thus equality holds throughout. It follows that
$$\sH(f \cdot \alpha / F' \cdot \alpha) = \sH(f_0 t \cdot \alpha / F' \cdot \alpha) = \sH(t \cdot \alpha / f_0^{-1} F' \cdot \alpha) = \sH(t \cdot \alpha / \alpha).$$

Let $i \in \{-1, 1\}$ be such that $t^i \in S$. So we have $R_S(F) = R_S(F') + t^i$. By clause (i) of Lemma \ref{LEM SHAN} and the inductive hypothesis we have
$$\sH(F \cdot \alpha) = \sH(F' \cdot \alpha \vee f \cdot \alpha) = \sH(F' \cdot \alpha) + \sH(f \cdot \alpha / F' \cdot \alpha)$$
$$= \sH(\alpha) + \zeta(R_S(F')) + \sH(t \cdot \alpha / \alpha) = \sH(\alpha) + \zeta(R_S(F')) + \sH(t^i \cdot \alpha / \alpha)$$
$$= \sH(\alpha) + \zeta(R_S(F')) + \zeta(t^i) = \sH(\alpha) + \zeta(R_S(F)).$$
Induction now completes the proof.
\end{proof}

Just as Shannon entropies simplify for Markov processes, so does the formula for f-invariant entropy. In fact within the context of finitely generated free groups and generating partitions with finite Shannon entropy, this provides yet another characterization of Markov processes.

\begin{thm}[Bowen, \cite{B10d}] \label{THM MARKFINV}
Let $G$ be a finitely generated free group, let $S$ be a free generating set for $G$, let $G$ act on $(X, \mu)$, and let $\alpha$ be a countable measurable partition of $X$. Assume that $\alpha$ is generating and has finite Shannon entropy. Then $X$ is a $(S, \alpha)$-Markov process if and only if
$$f_G(X, \mu) = F_G(X, \mu, S, \alpha) = (1 - 2 r) \sH(\alpha) + \sum_{s \in S} \sH(s \cdot \alpha \vee \alpha),$$
where $r$ is the rank of $G$.
\end{thm}

We now prove an important lemma which will significantly simplify some of our later proofs. The lemma below is also quite pleasing as it affirms the truth of something which one would intuitively expect. The usefulness of this lemma should extend beyond our work here.

\begin{lem} \label{LEM INTUIT}
Let $G$ be a free group acting on a probability space $(X, \mu)$. Assume that $X$ is a $(S, \alpha)$-Markov process where $\sH(\alpha) < \infty$. Let $U, V, W \subseteq G$ with $U$ finite. If $V$ right $S$-separates $(U, W)$ then
$$\sH(U \cdot \alpha / (W \cup V) \cdot \alpha) = \sH(U \cdot \alpha / V \cdot \alpha).$$
\end{lem}

\begin{proof}
First suppose that $U = \{u\}$ is a singleton and that $V$ is finite. Partially order $V$ so that $v_1 \preceq v_2$ if and only if the unique right $S$-path from $v_1$ to $u$ traverses $v_2$, or equivalently $v_1 \preceq v_2$ if and only if $\RP_S(v_1, u) \subseteq \RP_S(v_2, u)$. Since $V$ is finite, there are a finite number of $\preceq$-maximal elements of $V$. Say the $\preceq$-maximal elements are $v_1, v_2, \ldots, v_n$. Set $V_0 = \{v_1, v_2, \ldots, v_n\}$. Then we have
$$V \subseteq \RP_S(V, u) = \RP_S(V_0, u).$$
We claim that if $W \subseteq \RP_S(V_0, u)$ is finite and for each $1 \leq i \leq n$ the set $W \cap \RP_S(v_i, u)$ is right $S$-connected and contains $v_i$ then
$$\sH(u \cdot \alpha / (W \cup V_0) \cdot \alpha) = \sH(u \cdot \alpha / V_0 \cdot \alpha).$$
We prove this claim by induction on the cardinality of $W$. Notice that these conditions imply that $V_0 \subseteq W$. If $|W| = |V_0|$ then $W = V_0$ and the claim is clear. Now suppose the claim holds whenever $|V_0| \leq |W| \leq q$. Let $V_0 \subseteq W \subseteq \RP_S(V_0, u)$ be such that $|W| = q + 1$ and and such that for each $1 \leq i \leq n$ the set $W \cap \RP_S(v_i, u)$ is right $S$-connected and contains $v_i$. Pick $1 \leq i \leq n$ with $|W \cap \RP_S(v_i, u)| \geq 2$. Let $w \in W \cap \RP_S(v_i, u)$ be right $S$-furthest from $v_i$. Set $W' = W \setminus \{w\}$. Since $w, v_i \in W \cap \RP_S(v_i, u)$ and $W \cap \RP_S(v_i, u)$ is right $S$-connected, there must be $z \in W'$ which is right $S$-adjacent to $w$. Since we chose $w$ to be right $S$-furthest from $v_i$ we have
$$z \in W' \cup V_0 \subseteq \{u\} \cup W' \cup V_0 \subseteq \RP_S(z, w).$$
Therefore
$$\alpha \leq z^{-1} (W' \cup V_0) \cdot \alpha \leq z^{-1} u \cdot \alpha \vee z^{-1}(W' \cup V_0) \cdot \alpha \leq \RP_S(1_G, z^{-1} w) \cdot \alpha.$$
As $z^{-1} w \in S \cup S^{-1}$ and $X$ is a $(S, \alpha)$-Markov process, by clause (iii) of Lemma \ref{LEM SHAN} we have that
$$\sH(z^{-1} w \cdot \alpha / \alpha) = \sH(z^{-1} w \cdot \alpha / \RP_S(1_G, z^{-1} w) \cdot \alpha)$$
$$\leq \sH(z^{-1} w \cdot \alpha / z^{-1} u \cdot \alpha \vee z^{-1}(W' \cup V_0) \cdot \alpha) \leq \sH(z^{-1} w \cdot \alpha / z^{-1}(W' \cup V_0) \cdot \alpha)$$
$$\leq \sH(z^{-1} w \cdot \alpha / \alpha).$$
So equality holds throughout. It follows that
$$\sH(w \cdot \alpha / (W' \cup V_0) \cdot \alpha) = \sH(z^{-1} w \cdot \alpha / z^{-1}(W' \cup V_0) \cdot \alpha)$$
$$= \sH(z^{-1} w \cdot \alpha / z^{-1} u \cdot \alpha \vee z^{-1}(W' \cup V_0) \cdot \alpha) = \sH(w \cdot \alpha / u \cdot \alpha \vee (W' \cup V_0) \cdot \alpha).$$
By clause (ii) of Lemma \ref{LEM SHAN} and the inductive hypothesis we have
$$\sH(u \cdot \alpha / (W \cup V_0) \cdot \alpha) = \sH(u \cdot \alpha / w \cdot \alpha \vee (W' \cup V_0) \cdot \alpha)$$
$$= \sH(w \cdot \alpha / u \cdot \alpha \vee (W' \cup V_0) \cdot \alpha) + \sH(u \cdot \alpha / (W' \cup V_0) \cdot \alpha) - \sH(w \cdot \alpha / (W' \cup V_0) \cdot \alpha)$$
$$= \sH(u \cdot \alpha / (W' \cup V_0) \cdot \alpha) = \sH(u \cdot \alpha / V_0 \cdot \alpha).$$
So by induction we have that
$$\sH(u \cdot \alpha / (W \cup V_0) \cdot \alpha) = \sH(u \cdot \alpha / V_0 \cdot \alpha)$$
whenever $W$ is finite, $V_0$ right $S$-separates $(u, W)$, and for each $1 \leq i \leq n$ the set $W \cap \RP_S(v_i, u)$ is right $S$-connected and contains $v_i$.

Now suppose that $W$ is finite and that $V$ right $S$-separates $(u, W)$, where $V$ and $u$ are the same as in the previous paragraph. Then there is a finite set $W'$ such that $W \cup V \subseteq W'$, $V_0$ right $S$-separates $(u, W')$, and for each $1 \leq i \leq n$ the set $W' \cap \RP_S(v_i, u)$ is right $S$-connected and contains $v_i$. It follows from the previous paragraph that
$$\sH(u \cdot \alpha / V_0 \cdot \alpha) = \sH(u \cdot \alpha / (W' \cup V_0) \cdot \alpha) \leq \sH(u \cdot \alpha / (W \cup V) \cdot \alpha)$$
$$\leq \sH(u \cdot \alpha / V \cdot \alpha) \leq \sH(u \cdot \alpha / V_0 \cdot \alpha).$$
So equality holds throughout and
$$\sH(u \cdot \alpha / (W \cup V) \cdot \alpha) = \sH(u \cdot \alpha / V \cdot \alpha).$$
We conclude that for any $u \in G$ and any two finite sets $V, W \subseteq G$ with $V$ right $S$-separating $(u, W)$ we have
$$\sH(u \cdot \alpha / (W \cup V) \cdot \alpha) = \sH(u \cdot \alpha / V \cdot \alpha).$$

Now let $u \in G$ and $V, W \subseteq G$ be such that $V$ right $S$-separates $(u, W)$. We allow $V$ and $W$ to be infinite. Let $(W_n)_{n \in \N}$ be an increasing sequence of finite subsets of $W$ with $\bigcup_{n \in \N} W_n = W$. For each $n$ let $V_n \subseteq V$ be a finite set such that $V_n$ right $S$-separates $(u, W_n)$. By enlarging the $V_n$'s if necessary, we may suppose that they are increasing and union to $V$. So by clause (iv) of Lemma \ref{LEM SHAN} we have
$$\sH(u \cdot \alpha / (W \cup V) \cdot \alpha) = \lim_{n \rightarrow \infty} \sH(u \cdot \alpha / (W_n \cup V_n) \cdot \alpha) = \lim_{n \rightarrow \infty} \sH(u \cdot \alpha / V_n \cdot \alpha) = \sH(u \cdot \alpha / V \cdot \alpha).$$

Now let $U, V, W \subseteq G$ be such that $U$ is finite and $V$ right $S$-separates $(U, W)$. Enumerate $U$ as $U = \{u_1, u_2, \ldots, u_n\}$. Notice that $V_i = V \cup \{u_1, u_2, \ldots, u_{i-1}\}$ right $S$-separates $(u_i, W)$ for each $1 \leq i \leq n$. Clause (ii) of Lemma \ref{LEM SHAN} together with the previous paragraph give
$$\sH(U \cdot \alpha / (W \cup V) \cdot \alpha) = \sum_{i = 1}^n \sH(u_i \cdot \alpha / (W \cup V_i) \cdot \alpha) = \sum_{i = 1}^n \sH(u_i \cdot \alpha / V_i \cdot \alpha) = \sH(U \cdot \alpha / V \cdot \alpha).$$
This completes the proof.
\end{proof}

In order to prove that $f_H(X, \mu) = |G : H| \cdot f_G(X, \mu)$ for Markov processes $G \acts (X, \mu)$, we will find it convenient to work with a single partition $\beta$ which is generating for both $G \acts (X, \mu)$ and $H \acts (X, \mu)$. We will also want $\beta$ to be a Markov partition for $G \acts (X, \mu)$. We therefore need to know how much flexibility there is in choosing Markov partitions. This is addressed by the following lemma due to Bowen.

\begin{lem}[Bowen, \cite{B10d}] \label{LEM BIGPART}
Let $G$ be a free group acting on a probability space $(X, \mu)$. Suppose that $X$ is a $(S, \alpha)$-Markov process with $\sH(\alpha) < \infty$. Then $X$ is a $(S, \Delta \cdot \alpha)$-Markov process for every finite left $S$-connected set $\Delta \subseteq G$ containing the identity.
\end{lem}

The lemma states that it is sufficient for $\Delta$ to be left $S$-connected. We remark that in general it is necessary that $\Delta$ be left $S$-connected. Consider a Bernoulli shift $(K^G, \mu^G)$ and let $\alpha$ be the canonical partition. Then $K^G$ is a $(S, \alpha)$-Markov process. If $\Delta \subseteq G$ is not left $S$-connected, then one can use Theorem \ref{THM MARKFINV} to show that $K^G$ is not a $(S, \Delta \cdot \alpha)$-Markov process.

This lemma plays a crucial role in our main theorem, and so as a convenience to the reader we include a proof below. We remark that this proof is simpler and more intuitive than the proof in \cite{B10d} as here we rely on Lemma \ref{LEM INTUIT}.

\begin{proof}
Set $\beta = \Delta \cdot \alpha$. By Lemma \ref{LEM DEFN} it suffices to show that
$$\sH(s \cdot \beta / \RP_S(1_G, s) \cdot \beta) = \sH(s \cdot \beta / \beta)$$
for every $s \in S \cup S^{-1}$.

Fix $s \in S \cup S^{-1}$. Let $g \in \RP_S(1_G, s)$ and let $\delta \in \Delta$. Notice for $f \in G$, $f \in \RP_S(1_G, s)$ if and only if the reduced $S$-word representation of $f$ does not begin on the left with $s$. So if $g \delta \not\in \RP_S(1_G, s)$, then the reduced $S$-word representation of $\delta$ must begin with the reduced $S$-word representation of $g^{-1}$. So the reduced $S$-word representation of $g \delta$ is obtained from the reduced $S$-word representation of $\delta$ by removing an initial segment. Since $1_G \in \Delta$ and $\Delta$ is left $S$-connected, it follows that $g \delta \in \Delta$. Therefore
$$\RP_S(1_G, s) \cdot \Delta \subseteq \RP_S(1_G, s) \cup \Delta.$$
A similar argument shows that
$$s \Delta \subseteq \RP_S(s, 1_G) \cup \Delta.$$
Therefore $\Delta$ right $S$-separates $(s \Delta, \RP_S(1_G, s) \Delta)$. Since $\Delta \subseteq \RP_S(1_G, s) \cdot \Delta$, by Lemma \ref{LEM INTUIT} we have
$$\sH(s \cdot \beta / \RP_S(1_G, s) \cdot \beta) = \sH(s \Delta \cdot \alpha / \RP_S(1_G, s) \Delta \cdot \alpha)$$
$$= \sH(s \Delta \cdot \alpha / \Delta \cdot \alpha) = \sH(s \cdot \beta / \beta).$$
\end{proof}

In the next section, after we prove that $f_H(X, \mu) = |G : H| \cdot f_G(X, \mu)$ for Markov processes $G \acts (X, \mu)$, we will extend this relation to general actions by approximating by Markov processes. The precise tool we will need is described in the following definition.

\begin{defn}
Let $G$ be a finitely generated free group acting on a probability space $(X, \mu)$, let $S$ be a free generating set for $G$, and let $\alpha$ be a generating partition. A Borel probability measure $\mu'$ on $X$ is called a \emph{$(S, \alpha)$-Markov approximation to $\mu$} if $\mu'$ is $G$-invariant, $(X, \mu')$ is a $(S, \alpha)$-Markov process, and
$$\forall s \in S \cup S^{-1} \ \forall A_1, A_2 \in \alpha \ \mu'(A_1 \cap s \cdot A_2) = \mu(A_1 \cap s \cdot A_2).$$
\end{defn}

Markov approximations can be used to approximate f-invariant entropy, as the following simple lemma shows.

\begin{lem} \label{LEM MARKAPP1}
Let $G$ be a finitely generated free group acting on a probability space $(X, \mu)$. Let $S$ be a free generating set for $G$ and let $\alpha$ be a countable measurable partition of $X$ with $\sH(\alpha) < \infty$. If $\mu'$ is a $(S, \alpha)$-Markov approximation to $\mu$ then
$$F_G(X, \mu', S, \alpha) = F_G(X, \mu, S, \alpha).$$
\end{lem}

\begin{proof}
Since $\mu'(A_1 \cap s \cdot A_2) = \mu(A_1 \cap s \cdot A_2)$ for every $A_1, A_2 \in \alpha$ and $s \in S \cup S^{-1}$, we have
$$\sH_{\mu'}(\alpha) = \sH_\mu(\alpha) \text{ and } \forall s \in S \cup S^{-1} \ \sH_{\mu'}(s \cdot \alpha \vee \alpha) = \sH_\mu(s \cdot \alpha \vee \alpha).$$
So the lemma now immediately follows from the definition of $F_G(X, \cdot, S, \alpha)$.
\end{proof}

In general Markov approximations do not always exist, however if one is willing to replace $G \acts (X, \mu)$ with a measurably conjugate action $G \acts (Y, \nu)$, then one can arrange for Markov approximations to exist. When a Markov approximation does exist, it is unique \cite{B10d}. In order for Markov approximations to exist, it is sufficient to work within the setting of symbolic actions and canonical partitions.

\begin{defn}
Let $G$ be a countable group, and let $K$ be a countable set with the discrete topology. Let $K^G$ denote the set of all functions from $G$ to $K$ endowed with the product topology, and let $G$ act on $K^G$ by permuting coordinates:
$$\forall x \in K^G \ \forall g, h \in G \ (g \cdot x)(h) = x(g^{-1} h).$$
We call the action of $G$ on $K^G$ a \emph{symbolic action}. The \emph{canonical partition} of $K^G$ is $\alpha = \{A_k \: k \in K\}$, where $A_k = \{x \in K^G \: x(1_G) = k\}$.
\end{defn}

There is no loss in generality in working with symbolic actions, as the following lemma shows.

\begin{lem} \label{LEM SYMREP}
Let $G$ be a countable group acting on a probability space $(X, \mu)$, and let $\alpha$ be a generating partition. Then there exists a measurable map $\phi: X \rightarrow \alpha^G$ such that $\phi: (X, \mu) \rightarrow (\alpha^G, \phi_*(\mu))$ is a measure conjugacy and $\alpha = \phi^{-1}(\beta)$, where $\beta$ is the canonical partition of $\alpha^G$.
\end{lem}

\begin{proof}
Since $\alpha$ is generating, by definition we have that $\alpha$ is countable. Thus $G \acts \alpha^G$ is a symbolic action. Define $\zeta: X \rightarrow \alpha$ by letting $\zeta(x)$ be the unique $A \in \alpha$ with $x \in A$. We define a map $\phi: X \rightarrow \alpha^G$ by
$$\phi(x)(g) = \zeta(g^{-1} \cdot x).$$
The function $\phi$ is $G$-equivariant since
$$\phi(h \cdot x)(g) = \zeta(g^{-1} \cdot h \cdot x) = \phi(x)(h^{-1} g) = [h \cdot \phi(x)](g).$$
Let $\nu$ be the pushforward measure, $\nu = \phi_*(\mu)$. Then $\phi$ is an isomorphism between $(X, \mu)$ and $(\alpha^G, \nu)$ since $\alpha$ is generating and both of these probability spaces are standard Borel probability spaces. Let $\beta$ be the canonical partition of $\alpha^G$. Write $\beta = \{B_A \: A \in \alpha\}$ where $B_A = \{y \in \alpha^G \: y(1_G) = A\}$. Clearly $\phi(A) \subseteq B_A$ for every $A \in \alpha$. Therefore $A \subseteq \phi^{-1}(B_A)$ for each $A \in \alpha$. Since both $\alpha$ and $\phi^{-1}(\beta)$ are partitions of $X$, it follows that $\phi^{-1}(\beta) = \alpha$.
\end{proof}

\begin{thm}[Bowen, \cite{B10d}] \label{THM MARKAPP}
Let $G$ be a finitely generated free group, and let $S$ be a free generating set for $G$. If $G \acts K^G$ is a symbolic action, $\mu$ is a $G$-invariant Borel probability measure, and $\alpha$ is the canonical partition of $K^G$, then there exists a unique $G$-invariant Borel probability measure $\mu'$ on $K^G$ which is a $(S, \alpha)$-Markov approximation to $\mu$.
\end{thm}

In Appendix A of \cite{BG}, Bowen and Gutman show that a stronger property holds. With the same notation and assumptions as in the previous theorem, they showed that if $B_S(n)$ denotes the $S$-ball of radius $n$ centered on the identity, then there exists a unique $G$-invariant Borel probability measure $\mu'$ on $K^G$ which is a $(S, B_S(n) \cdot \alpha)$-Markov approximation to $\mu$. Their result is sufficient for our needs in the next section, however we will obtain tighter bounds in our corollaries by proving the following.

\begin{lem} \label{LEM MARKAPP2}
Let $G$ be a finitely generated free group, and let $S$ be a free generating set for $G$. Let $G \acts K^G$ be a symbolic action, let $\mu$ be a $G$-invariant Borel probability measure, and let $\alpha$ be the canonical partition of $K^G$. If $U \subseteq G$ is finite, left $S$-connected, and contains the identity, then there exists a unique $G$-invariant Borel probability measure $\mu'$ on $K^G$ which is a $(S, U \cdot \alpha)$-Markov approximation to $\mu$.
\end{lem}

\begin{proof}
Write $\alpha = \{A_k \: k \in K\}$ where
$$A_k = \{x \in K^G \: x(1_G) = k\}.$$
Set $\beta = U \cdot \alpha$ and write $\beta = \{B_z \: z \in K^U\}$, where for $z \in K^U$
$$B_z = \bigcap_{u \in U} u \cdot A_{z(u)}.$$
Consider the set $Y \subseteq (K^U)^G$ defined by
$$y \in Y \Longleftrightarrow \forall s \in S \cup S^{-1} \ \forall g \in G \ B_{y(g)} \cap s \cdot B_{y(g s)} \neq \varnothing.$$
Notice that $Y$ is $G$-invariant and closed. We claim that $G \acts K^G$ is topologically conjugate to $G \acts Y$, where $Y \subseteq (K^U)^G$ has the subspace topology. Define $\phi: K^G \rightarrow (K^U)^G$ by
$$\phi(x)(g)(u) = x(g u).$$
Since $U$ is finite and the map $x \mapsto x(g u)$ is continuous, $\phi$ is also continuous. If $x_1 \neq x_2 \in K^G$, then there is $g \in G$ with $x_1(g) \neq x_2(g)$. Hence $\phi(x_1)(g)(1_G) \neq \phi(x_2)(g)(1_G)$, so $\phi(x_1) \neq \phi(x_2)$, and therefore $\phi$ is injective. We have that $\phi$ is $G$-equivariant since
$$\phi(h \cdot x)(g)(u) = [h \cdot x](g u) = x(h^{-1} g u) = \phi(x)(h^{-1} g)(u) = [h \cdot \phi(x)](g)(u).$$
Also, if $g \in G$ and $s \in S \cup S^{-1}$ then for every $u \in U$
$$g^{-1} \cdot x \in u \cdot A_{x(g u)} \text{ and } g^{-1} \cdot x \in s u \cdot A_{x(g s u)}.$$
Therefore
$$g^{-1} \cdot x \in \bigcap_{u \in U} u \cdot A_{x(g u)} = B_{\phi(x)(g)} \text{ and } g^{-1} \cdot x \in \bigcap_{u \in U} s u \cdot A_{x(g s u)} = s \cdot B_{\phi(x)(g s)}.$$
So
$$B_{\phi(x)(g)} \cap s \cdot B_{\phi(x)(g s)} \supseteq \{g^{-1} \cdot x\} \neq \varnothing$$
and thus $\phi$ maps $K^G$ into $Y$.

It remains to show that $\phi$ maps $K^G$ onto $Y$ and $\phi^{-1}$ is continuous. Fix $y \in Y$. Define $x \in K^G$ by $x(g) = y(g)(1_G)$. We claim that $\phi(x) = y$. If so then $\phi$ will map onto $Y$ and $\phi^{-1}$ will be continuous, completing the proof that $\phi$ is a topological conjugacy. By  the definition of $\phi$ and of $x$ we have $\phi(x)(g)(u) = x(g u) = y(g u)(1_G)$. So $y(g)(u) = \phi(x)(g)(u)$ if and only if $y(g)(u) = y(g u)(1_G)$. Thus it suffices to show that $y(g u)(1_G) = y(g)(u)$ for every $g \in G$ and $u \in U$. First, we claim that if $u \in U$, $s \in S \cup S^{-1}$, and $s \cdot u \in U$, then $y(g)(s u) =  y(g s)(u)$. By the definition of $Y$ we have
$$\varnothing \neq B_{y(g)} \cap s \cdot B_{y(g s)} \subseteq s u \cdot A_{y(g)(s u)} \cap s u \cdot A_{y(g s)(u)}.$$
Since $\alpha = \{A_k \: k \in K\}$ is a partition, it immediately follows that $A_{y(g)(s u)} = A_{y(g s)(u)}$ and hence $y(g)(s u) = y(g s)(u)$. Now fix $u \in U$ and let $u = s_1 s_2 \cdots s_n$ be the reduced $S$-word representation of $u$, where each $s_i \in S \cup S^{-1}$. Since $U$ is left $S$-connected and contains the identity, we have that $s_i s_{i+1} \cdots s_n \in U$ for every $1 \leq i \leq n$. Furthermore, $1_G \in U$ by assumption. By the previous claim we have that
$$y(g)(u) = y(g)(s_1 s_2 \cdots s_n) = y(g s_1)(s_2 \cdots s_n)$$
$$= \cdots = y(g s_1 \cdots s_n)(1_G) = y(g u)(1_G).$$
Thus $\phi(x) = y$ so $\phi$ maps $K^G$ onto $Y$ and $\phi^{-1}$ is continuous. We conclude that $K^G$ and $Y$ are topologically conjugate via $\phi$.

Now we prove the lemma. Since $\phi: K^G \rightarrow Y \subseteq (K^U)^G$ is a topological conjugacy, it induces a measure conjugacy between $G \acts (K^G, \mu)$ and $G \acts ((K^U)^G, \nu)$, where $\nu = \phi_*(\mu)$ is the pushforward measure (so $\nu$ is supported on $Y$). Let $\xi = \{C_z \: z \in K^U\}$ be the canonical partition of $(K^U)^G$, where $C_z = \{y \in (K^U)^G \: y(1_G) = z\}$. If $z \in K^U$ and
$$x \in B_z = \bigcap_{u \in U} u \cdot A_{z(u)}$$
then $\phi(x)(1_G)(u) = x(u) = z(u)$ for every $u \in U$. Thus $\phi(B_z) \subseteq C_z$. Since $\phi$ is injective it follows that $\beta = \phi^{-1}(\xi)$. By Theorem \ref{THM MARKAPP}, there is a $G$-invariant Borel probability measure $\lambda$ on $(K^U)^G$ which is a $(S, \xi)$-Markov approximation to $\nu$.

We claim that the support of $\lambda$ is contained within $Y$. Fix $w \in (K^U)^G \setminus Y$. By the definition of $Y$, there are $g \in G$ and $s \in S \cup S^{-1}$ with
$$B_{w(g)} \cap s \cdot B_{w(g s)} = \varnothing.$$
Consider the open set
$$V = \{z \in (K^U)^G \: z(g) = w(g), \ z(g s) = w(g s)\}.$$
Then $w \in V$ and $V \cap Y = \varnothing$. It suffices to show that $\lambda(V) = 0$. We have that $g^{-1} \cdot V \in \xi \vee s \cdot \xi$, so
$$\lambda(V) = \lambda(g^{-1} \cdot V) = \nu(g^{-1} \cdot V) = \phi_*(\mu)(g^{-1} \cdot V) = \mu(\phi^{-1}(g^{-1} \cdot V)).$$
However, since $\phi$ maps $K^G$ into $Y$, we have $\phi^{-1}(g^{-1} \cdot V) = \varnothing$. Thus $\lambda(V) = 0$ as claimed.

Since the support of $\lambda$ is contained within the image of the topological conjugacy $\phi$, we have that $\phi$ induces a measure conjugacy between $((K^U)^G, \lambda)$ and $(K^G, \phi^{-1}_*(\lambda))$. Set $\mu' = \phi^{-1}_*(\lambda)$. Then $\mu'$ is a $G$-invariant Borel probability measure on $K^G$. Since $\lambda$ is a $(S, \xi)$-Markov approximation to $\nu$, by applying $\phi^{-1}$ we get that $\mu'$ is a $(S, \beta)$-Markov approximation to $\mu$. The measure $\mu'$ is unique by \cite[Theorem 7.1]{B10d}. This completes the proof as $\beta = U \cdot \alpha$.
\end{proof}

\section{Subgroups and f-invariant entropy} \label{SEC SUBFORM}

In this section we prove the main theorem and deduce some of its corollaries. Our goal is to first establish the main theorem in the context of Markov processes and then use Markov approximations to extend the result to general actions. Our first step is to show that if $G \acts (X, \mu)$ is a Markov process and $H \leq G$ is a subgroup of finite index, then $H \acts (X, \mu)$ is a Markov process as well. The difficulty in showing this is that the characterization of Markov processes delicately depends on both the choice of a free generating set for the group and on the choice of a generating partition for the action.

\begin{thm} \label{THM SUBMARK}
Let $G$ be a free group, let $G$ act on $(X, \mu)$, and let $H \leq G$ be a subgroup of finite index. If $G \acts (X, \mu)$ is a Markov process with a Markov partition having finite Shannon entropy, then $H \acts (X, \mu)$ is also a Markov process. In fact, if $G \acts (X, \mu)$ is a $(S, \alpha)$-Markov process with $\sH(\alpha) < \infty$ and $\Delta \subseteq G$ is any right $S$-connected transversal of the right $H$-cosets $\{H g \: g \in G\}$ with $1_G \in \Delta$, then there exists a free generating set $T$ for $H$ such that $H \acts (X, \mu)$ is a $(T, \Delta \cdot \alpha)$-Markov process.
\end{thm}

\begin{proof}
Assume that $G \acts (X, \mu)$ is a $(S, \alpha)$-Markov process. Following a construction of Schreier \cite[Theorem 2.9]{MKS}, we will pick a free generating set for $H$. Let $\Delta$ be a right $S$-connected transversal of the right $H$-cosets in $G$ with $1_G \in \Delta$. Define $r : G \rightarrow \Delta$ by letting $r(g) = \delta$ if and only if $H g = H \delta$. Define the cocycle $c: \Delta \times G \rightarrow H$ by
$$c(\delta, g) = \delta g \cdot r(\delta g)^{-1}.$$
Set $T = c(\Delta \times S) \setminus \{1_G\}$. We claim that $T$ is a free generating set for $H$.

Consider the directed and $S$-edge-labeled Schreier graph, $\Gamma$, of the right $H$-cosets in $G$. Specifically, the vertex set of $\Gamma$ is $\{H g \: g \in G\}$, and for every $g \in G$ and $s \in S$ there is an edge directed from $H g$ to $H g s$ labeled $s$. The right $S$-connected transversal $\Delta$ naturally gives rise to a spanning tree $\Lambda$ of $\Gamma$. Specifically, $\Lambda$ contains all of the vertices of $\Gamma$ and has an edge directed from $H \delta$ to $H \delta s$ labeled $s$ whenever $\delta, \delta s \in \Delta$ and $s \in S$. Clearly the fundamental group of $\Gamma$, $\pi_1(\Gamma)$, is naturally group isomorphic to $H$. Let $\phi: \pi_1(\Gamma) \rightarrow H$ be this group isomorphism. For each edge $e \in \E(\Gamma) \setminus \E(\Lambda)$, let $\ell_e$ be the simple loop in $\Lambda \cup \{e\}$ which begins and ends at the vertex $H$ and which traverses $e$ with positive orientation. By the van Kampen theorem, $\pi_1(\Gamma)$ is freely generated by the set $\{\ell_e \: e \in \E(\Gamma) \setminus \E(\Lambda)\}$. If $e \in \E(\Gamma) \setminus \E(\Lambda)$ is labeled by $s \in S$ and directed from $H \delta_1$ to $H \delta_2$, with $\delta_1, \delta_2 \in \Delta$, then
$$\phi(\ell_e) = \delta_1 s \delta_2^{-1} = \delta_1 s \cdot r(\delta_1 s)^{-1} = c(\delta_1, s) \in T.$$
So $\phi(\{\ell_e \: e \in \E(\Gamma) \setminus \E(\Lambda)\}) \subseteq T$. Now fix $t = \delta s \cdot r(\delta s)^{-1} \in T$. Since $t \neq 1_G$, we have that $\delta s \not\in \Delta$. Therefore the edge $e$ directed from $H \delta$ to $H \delta s$ and labeled $s$ is in $\E(\Gamma)$ but not in $\E(\Lambda)$. Thus $\ell_e$ is defined and clearly $\phi(\ell_e) = t$. Thus $\phi(\{\ell_e \: e \in \E(\Gamma) \setminus \E(\Lambda)\}) = T$. We conclude that $T$ freely generates $H$.

We claim that for $a \neq b \in H$, $a$ and $b$ are right $T$-adjacent if and only if $a \Delta \cup b \Delta$ is right $S$-connected. First suppose that $a$ and $b$ are right $T$-adjacent. Then we can swap $a$ and $b$ if necessary to find $t \in T$ with $b = a t$. Since $\Delta$ is right $S$-connected, so are both $a \Delta$ and $b \Delta$. So we only need to find a point in $a \Delta$ which is right $S$-adjacent to a point in $b \Delta$. Let $s \in S$ and $\delta_1, \delta_2 \in \Delta$ be such that $t = \delta_1 s \delta_2^{-1}$. Then we have that $a \delta_1 \in a \Delta$ is right $S$-adjacent to $a \delta_1 s = a t \delta_2 = b \delta_2 \in b \Delta$. Thus $a \Delta \cup b \Delta$ is right $S$-connected as claimed. Now suppose that $a \Delta \cup b \Delta$ is right $S$-connected. Then by swapping $a$ and $b$ if necessary we can find $s \in S$ and $\delta_1, \delta_2 \in \Delta$ with $a \delta_1 s = b \delta_2$. Notice that $H \delta_1 s = H \delta_2$ since $a, b \in H$, and therefore $r(\delta_1 s) = \delta_2$. We have
$$1_G \neq a^{-1} b = \delta_1 s \delta_2^{-1} = \delta_1 s \cdot r(\delta_1 s)^{-1} = c(\delta_1, s) \in T.$$
So for $t = \delta_1 s \delta_2^{-1}$ we have $a t = b$. Thus $a$ and $b$ are right $T$-adjacent as claimed.

It immediately follows from the previous paragraph that for $F \subseteq H$, $F$ is right $T$-connected if and only if $F \Delta$ is right $S$-connected. We claim that for $U, V, W \subseteq H$, $V$ right $T$-separates $(U, W)$ if and only if $V \Delta$ right $S$-separates $(U \Delta, W \Delta)$. First suppose that $V \Delta$ right $S$-separates $(U \Delta, W \Delta)$. Let $F \subseteq H$ be a right $T$-connected set with $U \cap F \neq \varnothing$ and $W \cap F \neq \varnothing$. Then $F \Delta$ is right $S$-connected and $U \Delta \cap F \Delta \neq \varnothing$ and $W \Delta \cap F \Delta \neq \varnothing$. So we must have that $V \Delta \cap F \Delta \neq \varnothing$. However, $V, F \subseteq H$ and since $\Delta$ is a transversal of the right $H$-cosets we have that $h_1 \Delta \cap h_2 \Delta \neq \varnothing$ if and only if $h_1 = h_2$. So we must have $F \cap V \neq \varnothing$. Therefore $V$ right $T$-separates $(U, W)$. Now suppose that $V$ right $T$-separates $(U, W)$. Let $F \subseteq G$ be a right $S$-connected set with $U \Delta \cap F \neq \varnothing$ and $W \Delta \cap F \neq \varnothing$. We must show that $V \Delta \cap F \neq \varnothing$. Define $F' \subseteq H$ by the rule
$$f \in F' \Longleftrightarrow f \Delta \cap F \neq \varnothing.$$
We have $F \subseteq F' \Delta$, so $U \Delta \cap F' \Delta \neq \varnothing$ and $W \Delta \cap F' \Delta \neq \varnothing$. Again, since $U, W, F' \subseteq H$ we have that $U \cap F' \neq \varnothing$ and $W \cap F' \neq \varnothing$. Furthermore, $F' \Delta$ is right $S$-connected since $F \subseteq F' \Delta$ is right $S$-connected and for every $f \in F'$ the set $f \Delta$ is right $S$-connected and meets $F$. This implies that $F'$ is right $T$-connected. Therefore $V \cap F' \neq \varnothing$. By the definition of $F'$, there is $v \in V$ with $v \in F'$ and hence $v \Delta \cap F \neq \varnothing$. So $V \Delta \cap F \neq \varnothing$ and we conclude that $V \Delta$ right $S$-separates $(U \Delta, W \Delta)$.

Now we show that $H \acts (X, \mu)$ is a $(T, \Delta \cdot \alpha)$-Markov process. We point out that $\Delta \cdot \alpha$ is a generating partition for $H \acts (X, \mu)$ since $G = H \Delta$ and $\alpha$ is generating for $G \acts (X, \mu)$. Set $\beta = \Delta \cdot \alpha$. Fix $t \in T \cup T^{-1}$. By Lemma \ref{LEM DEFN} it suffices to show that
$$\sH(t \cdot \beta / \RP_T(1_H, t) \cdot \beta) = \sH(t \cdot \beta / \beta).$$
We clearly have that $1_H$ right $T$-separates $(\RP_T(1_H, t), t)$, and so by the previous paragraph $\Delta$ right $S$-separates $(\RP_T(1_H, t) \Delta, t \Delta)$. Therefore by Lemma \ref{LEM INTUIT}
$$\sH(t \cdot \beta / \RP_T(1_H, t) \cdot \beta) = \sH(t \Delta \cdot \alpha / \RP_T(1_H, t) \Delta \cdot \alpha)$$
$$= \sH(t \Delta \cdot \alpha / \Delta \cdot \alpha) = \sH(t \cdot \beta / \beta).$$
This completes the proof.
\end{proof}

The following lemma is well known, but it also follows directly from the construction in the proof of the previous theorem.

\begin{lem}[Proposition I.3.9 \cite{LS77}] \label{LEM RANK}
Let $G$ be a finitely generated free group and let $r_G$ be the rank of $G$. If $H \leq G$ is of finite index, then the rank, $r_H$, of $H$ and index of $H$ are related by
$$r_H = |G : H| (r_G - 1) + 1.$$
\end{lem}

The following lemma deals with the function $R_S$ introduced in Definition \ref{DEFN REDGE}.

\begin{lem} \label{LEM SUBEDGE}
Let $G$ be a finitely generated free group, let $S$ be a free generating set for $G$, let $H \leq G$ be a subgroup of finite index, and let $\Delta$ be a right $S$-connected transversal of the right $H$-cosets $\{H g \: g \in G\}$ with $1_G \in \Delta$. If $T$ is the generating set for $H$ constructed in the proof of Theorem \ref{THM SUBMARK}, then
$$\sum_{t \in T} (R_S(t \Delta \cup \Delta) - R_S(\Delta)) = |T| \cdot R_S(\Delta) + \sum_{s \in S} (R_S(\Delta s \cup \Delta) - R_S(\Delta)).$$
\end{lem}

\begin{proof}
Let the functions $r: G \rightarrow \Delta$ and $c: \Delta \times G \rightarrow H$ be as defined in the proof of Theorem \ref{THM SUBMARK}. Consider the set $\Delta \times S$. We associate $(\delta, s) \in \Delta \times S$ with the edge $(\delta, \delta \cdot s)$ in the right $S$-Cayley graph of $G$. Call $(\delta, s) \in \Delta \times S$ \emph{internal} if $\delta s \in \Delta$, and call it external otherwise. Let $\Int(\Delta \times S)$ and $\Ext(\Delta \times S)$ denote the internal and external elements of $\Delta \times S$, respectively. The set $\Int(\Delta \times S)$ naturally produces a graph structure on $\Delta$. Since $\Delta$ is right $S$-connected, this graph is connected, and it is a tree since it is a subgraph of the right $S$-Cayley graph of $G$. It is well known that in any finite tree the number of edges is one less than the number of vertices \cite[I.2 Corollary 8]{BB}. So $|\Int(\Delta \times S)| = |\Delta| - 1$ and $|\Ext(\Delta \times S)| = |\Delta| \cdot |S| - |\Delta| + 1$. From the definition of $c$ it is readily observed that $c(\delta, s) = 1_G$ if and only if $(\delta, s) \in \Int(\Delta \times S)$. So by definition $T = c(\Ext(\Delta \times S))$. By the previous lemma,
$$|T| = |\Delta| (|S|-1) + 1 = |\Delta| |S| - |\Delta| + 1 = |\Ext(\Delta \times S)|.$$
Therefore $c$ is a bijection between $\Ext(\Delta \times S)$ and $T$.

For $t \in T$, let $(\delta_t, s_t) \in \Ext(\Delta \times S)$ be such that $c(\delta_t, s_t) = t$. Recall from the proof of Theorem \ref{THM SUBMARK} that $t \Delta \cup \Delta$ is right $S$-connected, but $t \Delta$ and $\Delta$ are disjoint. The unique right $S$-edge joining $\Delta$ to $t \Delta$ is $(\delta_t, \delta_t s_t)$. Therefore
$$R_S(t \Delta \cup \Delta) = 2 R_S(\Delta) + s_t.$$
We have
$$\sum_{t \in T} (R_S(t \Delta \cup \Delta) - R_S(\Delta)) = \sum_{t \in T} (R_S(\Delta) + s_t) = |T| \cdot R_S(\Delta) + \sum_{t \in T} s_t.$$
Fix $s \in S$. Since $c : \Ext(\Delta \times S) \rightarrow T$ is a bijection, we have
$$|\{t \in T \: s_t = s\}| \cdot s = |\{(\delta, s') \in \Ext(\Delta \times S) \: s' = s\}| \cdot s$$
$$= |\{\delta \in \Delta \: \delta s \not\in \Delta\}| \cdot s = R_S(\Delta s \cup \Delta) - R_S(\Delta).$$
Therefore
$$\sum_{t \in T} s_t = \sum_{s \in S} |\{t \in T \: s_t = s\}| \cdot s = \sum_{s \in S} (R_S(\Delta s \cup \Delta) - R_S(\Delta)),$$
completing the proof.
\end{proof}

It is somewhat surprising that left $S$-connected sets appear in Lemmas \ref{LEM FDECR}, \ref{LEM BIGPART}, and \ref{LEM MARKAPP2}, while right $S$-connected sets appear in Lemma \ref{LEM SHANEDGE} and Theorem \ref{THM SUBMARK}. In order to make use of these results simultaneously, we will need to work with bi-$S$-connected sets. The next lemma is tailored to this case. Later, in Lemma \ref{LEM BI}, we will see that bi-$S$-connected transversals to cosets of normal finite index subgroups always exist.

The following lemma is unique in that it requires bi-$S$-connected sets. This lemma appears to be false if bi-$S$-connected is replaced by left $S$-connected or right $S$-connected. This lemma is somewhat technical, but it is key to the proof of the main theorem. For notational simplicity, in the proof and statement of the lemma below we write $R_S(F)$ simply as $R(F)$, where $R_S(F)$ is as in Definition \ref{DEFN REDGE}.

\begin{lem} \label{LEM COMB}
Let $G$ be a finitely generated free group, let $S$ be a free generating set for $G$, and let $r = |S|$ be the rank of $G$. If $\Delta \subseteq G$ is finite, bi-$S$-connected, and contains the identity then
$$|\Delta| \cdot \sum_{s \in S} (R(s \Delta \cup \Delta) - R(\Delta)) = \sum_{s \in S} (R(\Delta s \cup \Delta) - R(\Delta)) + (|\Delta|(r - 1) + 1) \cdot R(\Delta).$$
\end{lem}

\begin{proof}
We first claim that every finite bi-$S$-connected $\Delta \subseteq G$ containing the identity satisfies the two equations:
\begin{equation}\label{EQ1}
|\Delta| \cdot \left( \sum_{s \in S} s \right) = \sum_{s \in S} \left( R(\Delta s \cup \Delta) - R(\Delta) \right) + R(\Delta);
\end{equation}
\begin{equation}\label{EQ2}
\sum_{s \in S} \left(R(s \Delta \cup \Delta) - R(\Delta) \right) = \left(\sum_{s \in S} s \right) + (r - 1) R(\Delta).
\end{equation}
Before proving this claim, we show how it implies the statement of the lemma. By using first Equation \ref{EQ2} and then Equation \ref{EQ1} we have
$$|\Delta| \cdot \sum_{s \in S} (R(s \Delta \cup \Delta) - R(\Delta)) = |\Delta| \cdot \left( \sum_{s \in S} s \right) + |\Delta| (r - 1) R(\Delta)$$
$$= \sum_{s \in S} (R(\Delta s \cup \Delta) - R(\Delta)) + (|\Delta| (r - 1) + 1) R(\Delta),$$
as in the statement of the lemma. Thus it suffices to prove that Equations \ref{EQ1} and \ref{EQ2} hold.

Consider the set $\bigcup_{s \in S} (\Delta s \cup \Delta)$. Since the right $S$-Cayley graph of $G$ is a tree and $\Delta$ is right $S$-connected, the collection of right $S$-edges of this set is precisely $\{ (\delta, \delta s) \: \delta \in \Delta, s \in S\}$ (in other words, every edge must have an endpoint in $\Delta$). Therefore
$$R \left( \bigcup_{s \in S} (\Delta s \cup \Delta) \right) = |\Delta| \left( \sum_{s \in S} s \right).$$
Since $S$ is a free generating set and $\Delta$ is right $S$-connected, we have that for any $s \neq t \in S$ the sets $\Delta s \setminus \Delta$ and $\Delta t \setminus \Delta$ are disjoint. Therefore
$$R \left( \bigcup_{s \in S} (\Delta s \cup \Delta) \right) = R(\Delta) + \sum_{s \in S} (R(\Delta s \cup \Delta) - R(\Delta)).$$
So Equation \ref{EQ1} follows.

To establish Equation \ref{EQ2}, we use induction on the number of elements of $\Delta$. Equation \ref{EQ2} holds if $\Delta$ is a singleton since then $R(s \Delta \cup \Delta) = s$ and $R(\Delta) = 0$. Now suppose that Equation \ref{EQ2} holds whenever $|\Delta| \leq k$. Consider a bi-$S$-connected set $\Delta$ with $1_G \in \Delta$ and $|\Delta| = k + 1$. Pick $\delta \in \Delta$ of maximal $S$-word length. Since $\Delta$ is bi-$S$-connected, there must be $g, h \in \Delta$ and $u, v \in S \cup S^{-1}$ with $\delta = u h = g v$. Since $\delta$ is of maximal $S$-word length, $h$ and $g$ must be of smaller $S$-word length. So the set $K = \Delta \setminus \{\delta\}$ is bi-$S$-connected, contains the identity, and has $k$ elements. If $s \in S$ and $s \neq u, u^{-1}$, then $s \Delta \cup \Delta$ is the disjoint union of $s K \cup K$ with $\{u h, s u h\}$. Since $s \neq u, u^{-1}$, $u h$ and $s u h$ cannot be right $S$-adjacent (we cannot have $u h t = s u h$ for $t \in S \cup S^{-1}$ since $s u h$ has longer $S$-word length than $u h$ and $s \neq u$). However there are right edges $(g, g v) = (g, u h)$ and $(s g, s g v) = (s g, s u h)$. Therefore for $s \in S$ with $s \neq u, u^{-1}$ we have
$$R(s \Delta \cup \Delta) = R(s K \cup K) + 2v.$$
If $u \in S$ then $u \Delta \cup \Delta$ is the disjoint union of $u K \cup K$ with $\{u^2 h\}$. If $u^{-1} \in S$ then $u^{-1} \Delta \cup \Delta$ is the disjoint union of $u^{-1} K \cup K$ with $\{u h\}$. In either case, we have
$$R(u^{\pm 1} \Delta \cup \Delta) = R(u^{\pm 1} K \cup K) + v,$$
where $u^{\pm 1}$ is chosen to be in $S$. Clearly $R(\Delta) = R(K) + v$. So by the inductive hypothesis we have
$$\sum_{s \in S} (R(s \Delta \cup \Delta) - R(\Delta)) = \sum_{s \in S} (R(s K \cup K) - R(K)) + (r - 1) \cdot v$$
$$= \left( \sum_{s \in S} s \right) + (r - 1) R(K) + (r - 1) \cdot v = \left( \sum_{s \in S} s \right) + (r - 1) R(\Delta).$$
By induction we conclude that Equation \ref{EQ2} holds for every finite bi-$S$-connected set $\Delta \subseteq G$ containing the identity. This completes the proof.
\end{proof}

If $G \acts (X, \mu)$ is a Markov process and $H \leq G$ is of finite index, then we would like to find a single partition which is a Markov partition for both $G \acts (X, \mu)$ and $H \acts (X, \mu)$. To apply Lemma \ref{LEM BIGPART} and Theorem \ref{THM SUBMARK}, we need to find a bi-$S$-connected transversal of the right $H$-cosets in $G$. Such a transversal exists at least when $H$ is normal in $G$, as the following lemma shows.

\begin{lem} \label{LEM BI}
Let $G$ be a finitely generated free group and let $S$ be a free generating set for $G$. If $K \lhd G$ is a normal subgroup then there exists a bi-$S$-connected transversal $\Delta$ of the cosets of $K$ in $G$ with $1_G \in \Delta$.
\end{lem}

\begin{proof}
Fix a total ordering $\preceq$ on $S \cup S^{-1}$. We extend $\preceq$ lexicographically to an ordering $\preceq_{\mathrm{lex}}$ on $(S \cup S^{-1})$-words of the same length. Specifically, if $x = x_1 x_2 \cdots x_n$ and $y = y_1 y_2 \cdots y_n$ are two $(S \cup S^{-1})$-words of common length $n$, then $x \preceq_{\mathrm{lex}} y$ if and only if $x = y$ or $x_i \prec y_i$ for the first $i$ where $x_i \neq y_i$. For $g \in G$, let $W_S(g)$ denote the reduced $S$-word representation of $g$. We define a well ordering, $\leq$, on $G$ as follows. For $g, h \in G$ we define $g \leq h$ if and only if the following two conditions hold:
\begin{enumerate}
\item[\rm (1)] the $S$-word length of $g$ is less than or equal to the $S$-word length of $h$;
\item[\rm (2)] if $g$ and $h$ have the same $S$-word-length, then $W_S(g) \preceq_{\mathrm{lex}} W_S(h)$.
\end{enumerate}
The ordering $\leq$ of $G$ has the following properties:
\begin{enumerate}
\item[\rm (i)] $\leq$ is a well ordering, that is, every non-empty subset of $G$ has a $\leq$-least element;
\item[\rm (ii)] if $W_S(g)$ ends with $s$ then $g \leq h \Longrightarrow g s^{-1} \leq h s^{-1}$;
\item[\rm (iii)] if $W_S(h)$ does not begin with $s^{-1}$ then $g \leq h \Longrightarrow s g \leq s h$.
\end{enumerate}
We leave verification of these three properties to the reader. Next we define $\Delta$.

For a $K$-coset $g K$, define $r(g K)$ to be the $\leq$-least element of $g K$. Such an element exists by clause (i). Set $\Delta = \{ r(g K) \: g \in G \}$. Clearly $\Delta$ is a transversal of the $K$-cosets in $G$ and $1_G \in \Delta$. We claim that $\Delta$ is bi-$S$-connected.

\underline{Right $S$-connected.}  Fix $s \in S \cup S^{-1}$ and $g \in G$ which does not end with $s^{-1}$. Assume that $g s = \delta \in \Delta$. We must show that $g \in \Delta$. Set $\psi = r(K g)$. Note that $\psi \leq g$. We have $K \psi s = K g s = K \delta$, so by definition of $\Delta$ we have $\delta \leq \psi s$. Since $\delta$ ends with $s$ we have
$$g = \delta s^{-1} \leq (\psi s) s^{-1} = \psi \leq g$$
by property (ii). Therefore $g = \psi \in \Delta$.

\underline{Left $S$-connected.} Fix $s \in S \cup S^{-1}$ and $g \in G$ which does not begin with $s^{-1}$. Assume that $s g = \delta \in \Delta$. We must show that $g \in \Delta$. Set $\psi = r(g K)$ and notice that $\psi \leq g$. We have $s \psi K = s g K = \delta K$, so by definition of $\Delta$ we have $\delta \leq s \psi$. Since $g$ does not begin with $s^{-1}$ we have
$$\delta \leq s \psi \leq s g = \delta$$
by property (iii). Therefore $g = \psi \in \Delta$.
\end{proof}

We are now ready to fit the individual pieces together and prove the main theorem within the context of Markov processes and normal subgroups.

\begin{prop}
Let $G$ be a finitely generated free group acting on a probability space $(X, \mu)$. Assume that $(X, \mu)$ is a $(S, \alpha)$-Markov process where $\sH(\alpha) < \infty$. If $K \lhd G$ is of finite index then $f_K(X, \mu)$ is defined and
$$f_K(X, \mu) = |G : K| \cdot f_G(X, \mu).$$
\end{prop}

\begin{proof}
Apply Lemma \ref{LEM BI} to get a bi-$S$-connected set $\Delta$ which contains the identity and is a transversal of the $K$-cosets in $G$. Set $\beta = \Delta \cdot \alpha$. Since $\Delta$ is left $S$-connected and contains the identity, $G \acts (X, \mu)$ is a $(S, \beta)$-Markov process. Since $\Delta$ is right $S$-connected and contains the identity, $K \acts (X, \mu)$ is a $(T, \beta)$-Markov process, where $T$ is as constructed in the proof of Theorem \ref{THM SUBMARK}. Notice that $\sH(\beta) \leq |\Delta| \cdot \sH(\alpha) < \infty$ and therefore $f_K(X, \mu)$ is defined.

From Lemmas \ref{LEM RANK}, \ref{LEM SUBEDGE}, and \ref{LEM COMB} we obtain (below $r_G$ is the rank of $G$)
$$\sum_{t \in T} (R_S(t \Delta \cup \Delta) - R_S(\Delta)) = |T| \cdot R_S(\Delta) + \sum_{s \in S} (R_S(\Delta s \cup \Delta) - R_S(\Delta))$$
$$= (|G : K| (r_G - 1) + 1) \cdot R_S(\Delta) + \sum_{s \in S} (R_S(\Delta s \cup \Delta) - R_S(\Delta))$$
$$= (|\Delta| \cdot (r_G - 1) + 1) \cdot R_S(\Delta) + \sum_{s \in S} (R_S(\Delta s \cup \Delta) - R_S(\Delta))$$
$$= |\Delta| \cdot \sum_{s \in S} (R_S(s \Delta \cup \Delta) - R_S(\Delta))$$
$$= |G : K| \cdot \sum_{s \in S} (R_S(s \Delta \cup \Delta) - R_S(\Delta)).$$
From Lemma \ref{LEM SHANEDGE} it follows that
$$\sum_{t \in T} (\sH(t \cdot \beta \vee \beta) - \sH(\beta)) = |G : K| \cdot \sum_{s \in S} (\sH(s \cdot \beta \vee \beta) - \sH(\beta)).$$
So applying Theorem \ref{THM MARKFINV} to both $G \acts (X, \mu)$ and $K \acts (X, \mu)$ gives (below $r_K$ is the rank of $K$)
$$|G : K| \cdot f_G(X, \mu) = |G: K| (1 - 2 r_G) \sH(\beta) + |G : K| \cdot \sum_{s \in S} \sH(s \cdot \beta \vee \beta)$$
$$= |G: K| (1 - r_G) \sH(\beta) + |G : K| \cdot \sum_{s \in S} (\sH(s \cdot \beta \vee \beta) - \sH(\beta))$$
$$= (1 - r_K) \sH(\beta) + |G : K| \cdot \sum_{s \in S} (\sH(s \cdot \beta \vee \beta) - \sH(\beta))$$
$$= (1 - r_K) \sH(\beta) + \sum_{t \in T} (\sH(t \cdot \beta \vee \beta) - \sH(\beta)) = (1 - 2 r_K) \sH(\beta) + \sum_{t \in T} \sH(t \cdot \beta \vee \beta)$$
$$= f_K(X, \mu).$$
This completes the proof.
\end{proof}

\begin{cor} \label{COR MARKRATIO}
Let $G$ be a finitely generated free group acting on a probability space $(X, \mu)$. Assume that $(X, \mu)$ is a $(S, \alpha)$-Markov process where $\sH(\alpha) < \infty$. If $H \leq G$ is of finite index then $f_H(X, \mu)$ is defined and
$$f_H(X, \mu) = |G : H| \cdot f_G(X, \mu).$$
\end{cor}

\begin{proof}
We claim that $H$ contains a subgroup of finite index which is normal in $G$. To see this, consider the left $H$-cosets $\{g H \: g \in G\}$. Clearly $G$ acts on these cosets on the left, and this induces a homomorphism from $G$ into the finite symmetric group $\Sym(|G : H|)$. Let $K$ be the kernel of this homomorphism. Then $K$ is normal in $G$ and is of finite index. Furthermore, $K H = H$ and thus $K \leq H$. Since $G \acts (X, \mu)$ is a Markov process, by Theorem \ref{THM SUBMARK} we have that $H \acts (X, \mu)$ is a Markov process as well. Furthermore, the Markov partition for $H \acts (X, \mu)$ is of the form $\Delta \cdot \alpha$ where $\Delta$ is finite and hence $\sH(\Delta \cdot \alpha) < \infty$. So now the assumptions of the previous proposition are satisfied for both $K \lhd G$ and $K \lhd H$, so we have
$$f_H(X, \mu) = \frac{1}{|H : K|} \cdot f_K(X, \mu) = \frac{|G : K|}{|H : K|} \cdot f_G(X, \mu) = |G : H| \cdot f_G(X, \mu).$$
\end{proof}

We now use Markov approximations to obtain the main theorem. We remark that the use of Markov approximations is not as direct as one might expect. We can approximate the action of $G$ by Markov processes to obtain an inequality. However, we can not approximate the action of $H$ by Markov processes in order to obtain the reverse inequality because in general $G$ does not act measure preservingly on Markov approximations to the $H$ action.

\begin{thm} \label{THM MAIN}
Let $G$ be a finitely generated free group acting on a probability space $(X, \mu)$. Let $H \leq G$ be a subgroup of finite index, and let $H$ act on $X$ by restricting the action of $G$. If the f-invariant entropy is defined for either the $G$ action or the $H$ action, then it is defined for both actions and
$$f_H(X, \mu) = |G : H| \cdot f_G(X, \mu).$$
\end{thm}

\begin{proof}
If $\alpha$ is a finite Shannon entropy generating partition for $G \acts (X, \mu)$, then $\Delta \cdot \alpha$ is a finite Shannon entropy generating partition for $H \acts (X, \mu)$, where $\Delta$ is any transversal of the right $H$-cosets in $G$. Conversely, if $\alpha$ is a finite Shannon entropy generating partition for $H \acts (X, \mu)$, then it is also a finite Shannon entropy generating partition for $G \acts (X, \mu)$. Thus $f_G(X, \mu)$ is defined if and only if $f_H(X, \mu)$ is defined.

Assume that both $f_G(X, \mu)$ and $f_H(X, \mu)$ are defined. So there is a generating partition $\alpha$ for $G \acts (X, \mu)$ with $\sH(\alpha) < \infty$. Fix a free generating set $S$ for $G$.

We first show that $|G : H| \cdot f_G(X, \mu) \leq f_H(X, \mu)$. Let $T$ be any free generating set for $H$, let $V \subseteq G$ be any finite set satisfying $H V = G$, and let $W$ be any finite left $S$-connected set containing $T V \cup \{1_G\}$. Using Lemma \ref{LEM SYMREP}, fix a measure conjugacy $\phi: (X, \mu) \rightarrow (\alpha^G, \nu)$. Let $\xi$ be the canonical partition of $\alpha^G$ and recall that $\phi^{-1}(\xi) = \alpha$. By Lemma \ref{LEM MARKAPP2}, there is a $G$-invariant probability measure $\nu'$ on $\alpha^G$ which is a $(S, W \cdot \xi)$-Markov approximation to $\nu$. Then we have
\begin{center}
\begin{tabular}{l l l}
       & $|G : H| \cdot F_G(X, \mu, S, W \cdot \alpha)$ & \\
$=$    & $|G : H| \cdot F_G(\alpha^G, \nu, S, W \cdot \xi)$ & since $\phi$ is a measure conjugacy \\
$=$    & $|G : H| \cdot F_G(\alpha^G, \nu', S, W \cdot \xi)$ & by Lemma \ref{LEM MARKAPP1} \\
$=$    & $|G : H| \cdot f_G(\alpha^G, \nu')$ & by Theorem \ref{THM MARKFINV} \\
$=$    & $f_H(\alpha^G, \nu')$ & by Corollary \ref{COR MARKRATIO} \\
$\leq$ & $F_H(\alpha^G, \nu', T, V \cdot \xi)$ & since $V \cdot \xi$ is a generating partition \\
$=$    & $F_H(\alpha^G, \nu, T, V \cdot \xi)$ & since $T V \subseteq W$ and $\nu'$ and $\nu$ agree on $W \cdot \xi$ \\
$=$    & $F_H(X, \mu, T, V \cdot \alpha)$ & since $\phi$ is a measure conjugacy.
\end{tabular}
\end{center}
So $|G : H| \cdot F_G(X, \mu, S, W \cdot \alpha) \leq F_H(X, \mu, T, V, \cdot \alpha)$ whenever $T$ is any free generating set for $H$, $V$ is any finite set satisfying $H V = G$, and $W$ is any left $S$-connected set containing $T V \cup \{1_G\}$. Now for each $n \in \N$, let $W_n$ be a left $S$-connected finite set containing $T B_T(n) \Delta \cup \{1_G\}$, where $T$ is a free generating set for $H$, $B_T(n)$ is the $T$-ball of radius $n$ in $H$ centered on the identity, and $\Delta$ is a transversal of the right $H$-cosets. Then we have
$$f_H(X, \mu) = \lim_{n \rightarrow \infty} F_H(X, \mu, T, B_T(n) \Delta \cdot \alpha)$$
$$\geq \lim_{n \rightarrow \infty} |G : H| \cdot F_G(X, \mu, S, W_n \cdot \alpha) \geq |G : H| \cdot f_G(X, \mu).$$
This gives us one inequality. The reverse inequality will require more effort.

Let $\beta$ be any generating partition for $G \acts (X, \mu)$ with $\sH(\beta) < \infty$. Apply Lemma \ref{LEM SYMREP} to get a measure conjugacy $\phi: (X, \mu) \rightarrow (\beta^G, \nu)$. Let $\xi$ be the canonical partition of $\beta^G$ and recall that $\phi^{-1}(\xi) = \beta$. By Theorem \ref{THM MARKAPP}, we can let $\nu^*$ be the $(S, \xi)$-Markov approximation to $\nu$. Let $\Delta$ be a right $S$-connected transversal of the right $H$-cosets in $G$ with $1_G \in \Delta$, and let $T$ be the free generating set for $H$ constructed in the proof of Theorem \ref{THM SUBMARK}. We claim that
\begin{equation} \label{INEQ}
F_H(\beta^G, \nu^*, T, \Delta \cdot \xi) \geq F_H(\beta^G, \nu, T, \Delta \cdot \xi).
\end{equation}
We have
$$F_H(\beta^G, \nu^*, T, \Delta \cdot \xi) - F_H(\beta^G, \nu, T, \Delta \cdot \xi)$$
$$= \sH_{\nu^*}(\Delta \cdot \xi) + \sum_{t \in T} (\sH_{\nu^*}(t \Delta \cdot \xi / \Delta \cdot \xi) - \sH_{\nu^*}(\Delta \cdot \xi))$$
$$- \sH_\nu(\Delta \cdot \xi) - \sum_{t \in T}(\sH_\nu(t \Delta \cdot \xi / \Delta \cdot \xi) - \sH_\nu(\Delta \cdot \xi))$$
$$= (\sH_{\nu^*}(\Delta \cdot \xi) - \sH_\nu(\Delta \cdot \xi))$$
$$+ \sum_{t \in T} (\sH_{\nu^*}(t \Delta \cdot \xi / \Delta \cdot \xi) - \sH_{\nu^*}(\Delta \cdot \xi) - \sH_\nu(t \Delta \cdot \xi / \Delta \cdot \xi) + \sH_\nu(\Delta \cdot \xi)).$$
It will suffice to show that $\sH_{\nu^*}(\Delta \cdot \xi) - \sH_\nu(\Delta \cdot \xi) \geq 0$ and that for every $t \in T$
$$X_t = \sH_{\nu^*}(t \Delta \cdot \xi / \Delta \cdot \xi) - \sH_{\nu^*}(\Delta \cdot \xi) - \sH_\nu(t \Delta \cdot \xi / \Delta \cdot \xi) + \sH_\nu(\Delta \cdot \xi) \geq 0.$$
We prove these two inequalities in the following two paragraphs.

We will argue that $\sH_{\nu^*}(\Delta \cdot \xi) \geq \sH_\nu(\Delta \cdot \xi)$. Enumerate $\Delta$ as $\Delta = \{a_1, a_2, \ldots, a_n\}$ so that $a_1 = 1_G$ and for each $1 \leq i \leq n$ the set $K_i = \{a_1, a_2, \ldots, a_i\}$ is right $S$-connected. For each $2 \leq i \leq n$, let $b_i \in K_{i - 1}$ and $s_i \in S \cup S^{-1}$ be such that $a_i = b_i s_i$. By clauses (i) and (ii) of Lemma \ref{LEM SHAN} we have
$$\sH_{\nu^*}(\Delta \cdot \xi) - \sH_\nu(\Delta \cdot \xi)$$
$$= \sH_{\nu^*}(\xi) - \sH_\nu(\xi) + \sum_{i = 2}^n (\sH_{\nu^*}(a_i \cdot \xi / K_{i-1} \cdot \xi) - \sH_\nu(a_i \cdot \xi / K_{i-1} \cdot \xi))$$
$$= 0 + \sum_{i = 2}^n (\sH_{\nu^*}(b_i s_i \cdot \xi / K_{i-1} \cdot \xi) - \sH_\nu(b_i s_i \cdot \xi / K_{i-1} \cdot \xi))$$
$$= \sum_{i = 2}^n (\sH_{\nu^*}(s_i \cdot \xi / b_i^{-1} K_{i-1} \cdot \xi) - \sH_\nu(s_i \cdot \xi / b_i^{-1} K_{i-1} \cdot \xi))$$
$$= \sum_{i = 2}^n (\sH_{\nu^*}(s_i \cdot \xi / \xi) - \sH_\nu(s_i \cdot \xi / b_i^{-1} K_{i-1} \cdot \xi))$$
$$= \sum_{i = 2}^n (\sH_\nu(s_i \cdot \xi / \xi) - \sH_\nu(s_i \cdot \xi / b_i^{-1} K_{i-1} \cdot \xi)) \geq 0,$$
where for the second to last equality we use Lemma \ref{LEM INTUIT} and the fact that $b_i$ right $S$-separates $(a_i, K_{i-1})$ and hence $1_G$ right $S$-separates $(s_i, b_i^{-1} K_{i-1})$, and for the final inequality we use clause (iii) of Lemma \ref{LEM SHAN}.

Fix $t \in T$. We must show that $X_t \geq 0$. Let $\delta_1, \delta_2 \in \Delta$ and $s \in S$ be such that $t = \delta_1 s \delta_2^{-1}$. Recall that $t \Delta$ and $\Delta$ are disjoint but $t \Delta \cup \Delta$ is right $S$-connected. The unique right $S$-edge joining $\Delta$ to $t \Delta$ is $(\delta_1, \delta_1 s) = (\delta_1, t \delta_2)$. Let $\zeta: (\bigoplus_{s \in S} \Z \cdot s) \rightarrow \R$ be the linear extension of the map $s \mapsto \sH_{\nu^*}(s \cdot \xi / \xi)$. By Lemmas \ref{LEM SHAN} and \ref{LEM SHANEDGE} we have
$$\sH_{\nu^*}(t \Delta \cdot \xi / \Delta \cdot \xi) - \sH_{\nu^*}(\Delta \cdot \xi)$$
$$= \sH_{\nu^*}(t \Delta \cdot \xi \vee \Delta \cdot \xi) - 2 \cdot \sH_{\nu^*}(\Delta \cdot \xi)$$
$$= \sH_{\nu^*}(\xi) + \zeta(R_S(t \Delta \cup \Delta)) - 2 \cdot \sH_{\nu^*}(\xi) - 2 \cdot \zeta(R_S(\Delta))$$
$$= \zeta(s) - \sH_{\nu^*}(\xi) = \sH_{\nu^*}(s \cdot \xi / \xi) - \sH_{\nu^*}(\xi) = \sH_\nu(s \cdot \xi / \xi) - \sH_\nu(\xi).$$
Also, by Lemma \ref{LEM SHAN} we have
$$-\sH_\nu(t \Delta \cdot \xi / \Delta \cdot \xi) + \sH_\nu(\Delta \cdot \xi)$$
$$= -\sH_\nu(t \delta_2 \cdot \xi / \Delta \cdot \xi) - \sH_\nu(t \Delta \cdot \xi / t \delta_2 \cdot \xi \vee \Delta \cdot \xi) + \sH_\nu(\delta_2 \cdot \xi) + \sH_\nu(\Delta \cdot \xi / \delta_2 \cdot \xi)$$
Therefore
$$X_t = \sH_{\nu^*}(t \Delta \cdot \xi / \Delta \cdot \xi) - \sH_{\nu^*}(\Delta \cdot \xi) - \sH_\nu(t \Delta \cdot \xi / \Delta \cdot \xi) + \sH_\nu(\Delta \cdot \xi)$$
$$= \sH_\nu(s \cdot \xi / \xi) - \sH_\nu(\xi) - \sH_\nu(t \delta_2 \cdot \xi / \Delta \cdot \xi)$$
$$- \sH_\nu(t \Delta \cdot \xi / t \delta_2 \cdot \xi \vee \Delta \cdot \xi) + \sH_\nu(\delta_2 \cdot \xi) + \sH_\nu(\Delta \cdot \xi / \delta_2 \cdot \xi)$$
$$= \sH_\nu(s \cdot \xi / \xi) - \sH_\nu(t \delta_2 \cdot \xi / \Delta \cdot \xi) + \sH_\nu(\Delta \cdot \xi / \delta_2 \cdot \xi) - \sH_\nu(t \Delta \cdot \xi / t \delta_2 \cdot \xi \vee \Delta \cdot \xi)$$
$$= \sH_\nu(s \cdot \xi / \xi) - \sH_\nu(\delta_1 s \cdot \xi / \Delta \cdot \xi) + \sH_\nu(\Delta \cdot \xi / \delta_2 \cdot \xi) - \sH_\nu(t \Delta \cdot \xi / t \delta_2 \cdot \xi \vee \Delta \cdot \xi)$$
$$= \sH_\nu(s \cdot \xi / \xi) - \sH_\nu(s \cdot \xi / \delta_1^{-1} \Delta \cdot \xi) + \sH_\nu(\Delta \cdot \xi / \delta_2 \cdot \xi) - \sH_\nu(\Delta \cdot \xi / \delta_2 \cdot \xi \vee t^{-1} \Delta \cdot \xi).$$
This is non-negative by clause (iii) of Lemma \ref{LEM SHAN}, justifying the claim. Thus we conclude that Inequality \ref{INEQ} holds.

From the claim above it follows that
\begin{center}
\begin{tabular}{l l l}
    &    $|G : H| \cdot F_G(X, \mu, S, \beta)$ & \\
$=$ &    $|G : H| \cdot F_G(\beta^G, \nu, S, \xi)$ & since $\phi$ is a measure conjugacy \\
$=$ &    $|G : H| \cdot F_G(\beta^G, \nu^*, S, \xi)$ & by Lemma \ref{LEM MARKAPP1} \\
$=$ &    $|G : H| \cdot f_G(\beta^G, \nu^*)$ & by Theorem \ref{THM MARKFINV} \\
$=$ &    $f_H(\beta^G, \nu^*)$ & by Corollary \ref{COR MARKRATIO} \\
$=$ &    $F_H(\beta^G, \nu^*, T, \Delta \cdot \xi)$ & by Theorems \ref{THM SUBMARK} and \ref{THM MARKFINV}\\
$\geq$ & $F_H(\beta^G, \nu, T, \Delta \cdot \xi)$ & by Inequality \ref{INEQ} above \\
$=$ &    $F_H(X, \mu, T, \Delta \cdot \beta)$ & since $\phi$ is a measure conjugacy.
\end{tabular}
\end{center}
Thus, if $U \subseteq G$ is finite and non-empty, then by setting $\beta = U \cdot \alpha$ we obtain
$$|G : H| \cdot F_G(X, \mu, S, U \cdot \alpha) \geq F_H(X, \mu, T, \Delta U \cdot \alpha).$$
Therefore
$$|G : H| \cdot f_G(X, \mu) = \lim_{n \rightarrow \infty} |G : H| \cdot F_G(X, \mu, S, B_S(n) \cdot \alpha)$$
$$\geq \lim_{n \rightarrow \infty} F_H(X, \mu, T, \Delta B_S(n) \cdot \alpha) \geq f_H(X, \mu).$$
Thus $f_H(X, \mu) = |G : H| \cdot f_G(X, \mu)$.
\end{proof}

We now give an example to show that Theorem \ref{THM MAIN} is no longer true if one allows $H$ to have infinite index in $G$. When $|G : H| = \infty$, we take the equation $f_H(X, \mu) = |G : H| \cdot f_G(X, \mu)$ to mean that $f_G(X, \mu) = 0$ if $f_H(X, \mu)$ is finite, $f_H(X, \mu) = -\infty$ if $f_G(X, \mu) < 0$, and $f_H(X, \mu)$ is undefined if $f_G(X, \mu) > 0$ (since f-invariant entropy cannot attain the value $+ \infty$). The counter-example provided by the proposition below marks a difference between f-invariant entropy and Kolmogorov--Sinai entropy. For Kolmogorov--Sinai entropy, $h_H(X, \mu) = |G : H| \cdot h_G(X, \mu)$ whenever $H \leq G$, regardless if the index of $H$ in $G$ is finite or infinite.

\begin{prop}
There is a finitely generated free group $G$, a subgroup of infinite index $H \leq G$, and an action of $G$ on a probability space $(X, \mu)$ such that both $f_G(X, \mu)$ and $f_H(X, \mu)$ are defined but $f_H(X, \mu) \neq |G : H| \cdot f_G(X, \mu)$.
\end{prop}

\begin{proof}
Let $(X, \mu)$ be a standard probability space with $\mu$ supported on a countable set. Let $\alpha$ be a countable measurable partition of $X$ such that each atom of $\mu$ is a member of $\alpha$. Assume that $0 < \sH(\alpha) < \infty$. This can easily be arranged by having $\mu$ be the uniform probability measure on $n$ points, in which case $\sH(\alpha) = \log(n)$. We claim that for any finitely generated free group $G$ acting trivially on $X$ (fixing every point) we have $f_G(X, \mu) = (1 - r(G)) \cdot \sH(\alpha)$, where $r(G)$ is the rank of $G$. In fact, this follows immediately from the definition of f-invariant entropy. The partition $\alpha$ is trivially generating, and ignoring sets of measure zero we have $F \cdot \alpha = \alpha$ for every non-empty $F \subseteq G$. So $F_G(X, \mu, S, F \cdot \alpha) = (1 - r(G)) \cdot \sH(\alpha)$ for every non-empty $F \subseteq G$.

Now to prove the proposition, simply pick any non-cyclic finitely generated free group $G$ and any finitely generated free subgroup $H \leq G$ of infinite index. Then $- \infty < f_G(X, \mu) = (1 - r(G)) \cdot \sH(\alpha) < 0$ and $- \infty < f_H(X, \mu) = (1 - r(H)) \cdot \sH(\alpha) \leq 0$. Thus $f_H(X, \mu) \neq |G : H| \cdot f_G(X, \mu)$.
\end{proof}

It is unknown to the author if a less trivial counter-example exists. However, we observe the following constraint.

\begin{cor}
Let $G$ be a finitely generated free group, let $H \leq G$ be a non-trivial subgroup of infinite index, and let $G$ act on a probability space $(X, \mu)$. Suppose there are infinitely many finite index subgroups of $G$ containing $H$. If $f_H(X, \mu)$ is defined, then $f_G(X, \mu)$ is defined and $f_G(X, \mu) \leq 0$.
\end{cor}

Notice that if $H \leq K$ where $K$ is normal in $G$ and $G / K$ is residually finite, then there are infinitely many finite index subgroups of $G$ containing $H$.

\begin{proof}
Assume that $f_H(X, \mu)$ is defined. This assumption is equivalent to the existence of a countable partition $\alpha$ which is generating for $H \acts (X, \mu)$ and satisfies $\sH(\alpha) < \infty$. Clearly $\alpha$ is generating for $G \acts (X, \mu)$ and thus $f_G(X, \mu)$ is defined. Fix $\epsilon > 0$. Let $N \in \N$ be such that $\frac{1}{N} \cdot \sH(\alpha) < \epsilon$. Since $G$ is finitely generated, it has only finitely many subgroups having index less than or equal to $N$. Therefore there is a subgroup $\Gamma \leq G$ such that $N < |G : \Gamma| < \infty$ and $H \leq \Gamma$. Clearly $\alpha$ is a generating partition for $\Gamma \acts (X, \mu)$ and thus $f_\Gamma(X, \mu)$ is defined and satisfies $f_\Gamma(X, \mu) \leq \sH(\alpha)$. By Theorem \ref{THM MAIN} we have
$$f_G(X, \mu) = \frac{1}{|G : \Gamma|} \cdot f_\Gamma(X, \mu) \leq \frac{1}{N} \cdot \sH(\alpha) < \epsilon.$$
Letting $\epsilon$ tend to $0$ we obtain $f_G(X, \mu) \leq 0$. 
\end{proof}

In the corollary below we clarify and isolate the two inequalities obtained within the proof of Theorem \ref{THM MAIN}. This corollary can be thought of as a finitary version of the main theorem.

\begin{cor} \label{COR MAINFIN}
Let $G$ be a finitely generated free group acting on a probability space $(X, \mu)$. Let $H \leq G$ be a subgroup of finite index, let $S$ be a free generating set for $G$, and let $\alpha$ be a generating partition for $G \acts (X, \mu)$ with $\sH(\alpha) < \infty$. Then we have the following.
\begin{enumerate}
\item[\rm (i)] If $T$ is any free generating set for $H$, $V \subseteq G$ is any finite, non-empty set satisfying $H V = G$, and $W$ is any finite left $S$-connected set containing $T V \cup \{1_G\}$, then
$$F_H(X, \mu, T, V \cdot \alpha) \geq |G : H| \cdot F_G(X, \mu, S, W \cdot \alpha).$$
\item[\rm (ii)] If $\Delta$ is a right $S$-connected transversal of the right $H$-cosets in $G$ and contains the identity, $T$ is the free generating set for $H$ constructed in the proof of Theorem \ref{THM SUBMARK}, and $U \subseteq G$ is finite and non-empty then
$$F_H(X, \mu, T, \Delta U \cdot \alpha) \leq |G : H| \cdot F_G(X, \mu, S, U \cdot \alpha).$$
\end{enumerate}
\end{cor}

This finitary version of the main theorem provides us with new insight into Markov processes. It implies that in many circumstances the property of being a Markov process is independent of the choice of a free generating set for $G$.

\begin{cor} \label{COR MARKGEN}
Let $G$ be a finitely generated free group acting on a probability space $(X, \mu)$. Let $S_1$ and $S_2$ be two free generating sets for $G$. If $(X, \mu)$ is a $(S_1, \alpha_1)$-Markov process with $\sH(\alpha_1) < \infty$, then there exists a partition $\alpha_2$ with $\sH(\alpha_2) < \infty$ such that $(X, \mu)$ is a $(S_2, \alpha_2)$-Markov process.
\end{cor}

\begin{proof}
The key observation for this proof is that Corollary \ref{COR MAINFIN} does not require $H$ to be a proper subgroup of $G$. Let $W$ be any finite left $S_2$-connected set containing $S_1 \cup \{1_G\}$. We have
\begin{center}
\begin{tabular}{l l l}
       & $f_G(X, \mu)$ & \\
$=$    & $F_G(X, \mu, S_1, \alpha_1)$ & by Theorem \ref{THM MARKFINV} \\
$\geq$ & $F_G(X, \mu, S_2, W \cdot \alpha_1)$ & by clause (i) of Corollary \ref{COR MAINFIN} \\
$\geq$ & $f_G(X, \mu)$ & since $W \cdot \alpha_1$ is generating.
\end{tabular}
\end{center}
Therefore equality holds throughout. So it then follows from Theorem \ref{THM MARKFINV} that $G \acts (X, \mu)$ is a $(S_2, W \cdot \alpha_1)$-Markov process. Setting $\alpha_2 = W \cdot \alpha_1$ completes the proof since $\sH(\alpha_2) \leq |W| \cdot \sH(\alpha_1)$.
\end{proof}

The following corollary exhibits an interesting inequality involving f-invariant entropy. The author does not know how to obtain this inequality without applying Theorem \ref{THM MAIN}.

\begin{cor}
Let $G$ be a finitely generated free group acting on a probability space $(X, \mu)$, and let $\alpha$ be a generating partition having finite Shannon entropy. Then for any free generating set $S$ for $G$ and any finite right $S$-connected set $\Delta \subseteq G$ we have
$$f_G(X, \mu) \leq \frac{\sH(\Delta \cdot \alpha)}{|\Delta|} \leq \sH(\alpha).$$
\end{cor}

\begin{proof}
Fix a finite right $S$-connected $\Delta \subseteq G$. Since the action of $G$ is measure preserving, we can replace $\Delta$ with $\delta^{-1} \Delta$ if necessary in order to have $1_G \in \Delta$. We will define a right action, $*$, of $G$ on $\Delta$ as follows. Since $G$ is freely generated by $S$, it suffices to define how each $s \in S$ acts on $\Delta$. So fix $s \in S$ and $\delta \in \Delta$. If $\delta s \in \Delta$, then define $\delta * s = \delta s$. If $\delta s \not\in \Delta$, then let $k \geq 0$ be maximal with $\delta s^{-k} \in \Delta$ and then define $\delta * s = \delta s^{-k}$. This defines the right action of $G$ on $\Delta$. Since $\Delta$ is right $S$-connected and we defined $\delta * s = \delta s$ whenever $\delta s \in \Delta$, it follows that the action of $G$ on $\Delta$ is transitive. Let $H = \{g \in G \: 1_G * g = 1_G\}$ be the stabilizer of $1_G \in \Delta$. Then $H$ is a finite index subgroup of $G$ since $\Delta$ is finite. Furthermore, if $h \in H$ and $g \in G$ then
$$1_G * h g = (1_G * h) * g = 1_G * g.$$
Thus each point of $\Delta$ corresponds to a right $H$-coset. If $\delta \in \Delta$ then $1_G * \delta = \delta$ since $\Delta$ is right $S$-connected. Hence $\Delta$ is a right $S$-connected transversal of the right $H$-cosets in $G$. If $\alpha$ is a finite Shannon entropy generating partition for $G \acts (X, \mu)$ then $\Delta \cdot \alpha$ is a generating partition for $H \acts (X, \mu)$. So for any free generating set $T$ for $H$ we have
$$|\Delta| \cdot f_G(X, \mu) = |G : H| \cdot f_G(X, \mu) = f_H(X, \mu) \leq F_H(X, \mu, T, \Delta \cdot \alpha)$$
$$= \sH(\Delta \cdot \alpha) + \sum_{t \in T}(\sH(t \Delta \cdot \alpha/ \Delta \cdot \alpha) - \sH(\Delta \cdot \alpha)) \leq \sH(\Delta \cdot \alpha),$$
This establishes the first inequality. For the second inequality, it is easy to see that $\sH(\Delta \cdot \alpha) \leq |\Delta| \cdot \sH(\alpha)$.
\end{proof}

\section{Virtually free groups and virtual measure conjugacy} \label{SEC VIRT}

Our main theorem allows us to define f-invariant entropy for actions of finitely generated virtually free groups and also allows us to define a numerical invariant for virtual measure conjugacy. Recall that a group is \emph{virtually free} if it contains a free subgroup of finite index. Similarly, a group is \emph{virtually $\Z$} if it contains $\Z$ as a subgroup of finite index. To simplify discussion within this section, we will use the term ``virtually free'' to always mean virtually free but not virtually $\Z$.

\begin{cor} \label{COR NEWFINV}
Let $\Gamma$ be a finitely generated virtually free group acting on a probability space $(X, \mu)$. Assume that there is a generating partition for this action having finite Shannon entropy. If $G, H \leq \Gamma$ are finite index free subgroups, then $f_G(X, \mu)$ and $f_H(X, \mu)$ are defined and
$$\frac{1}{|\Gamma : G|} \cdot f_G(X, \mu) = \frac{1}{|\Gamma : H|} \cdot f_H(X, \mu).$$
Furthermore, if $\Gamma$ is itself free then the above common value is $f_\Gamma(X, \mu)$.
\end{cor}

\begin{proof}
Since $\Gamma$ is finitely generated and $G$ and $H$ are of finite index in $\Gamma$, we have that $G$ and $H$ are also finitely generated (\cite[Corollary IV.B.24]{Ha00}). Also, since they have finite index in $\Gamma$ and $\Gamma \acts (X, \mu)$ admits a finite Shannon entropy generating partition, the actions $G \acts (X, \mu)$ and $H \acts (X, \mu)$ also admit finite Shannon entropy generating partitions (by the same argument appearing in the first paragraph of the proof of Theorem \ref{THM MAIN}). Thus $f_G(X, \mu)$ and $f_H(X, \mu)$ are defined.

Consider the subgroup $K = G \cap H$. Clearly $K$ is a finite index subgroup of both $G$ and $H$. It follows from Theorem \ref{THM MAIN} that
$$\frac{1}{|\Gamma : G|} \cdot f_G(X, \mu) = \frac{1}{|\Gamma : G| \cdot |G : K|} \cdot f_K(X, \mu) = \frac{1}{|\Gamma : K|} \cdot f_K(X, \mu)$$
$$= \frac{1}{|\Gamma : H| \cdot |H : K|} \cdot f_K(X, \mu) = \frac{1}{|\Gamma : H|} \cdot f_H(X, \mu).$$
If $\Gamma$ is itself free then one can take $H = \Gamma$ to obtain
$$\frac{1}{|\Gamma : G|} \cdot f_G(X, \mu) = \frac{1}{|\Gamma : H|} \cdot f_H(X, \mu) = f_\Gamma(X, \mu).$$
This completes the proof.
\end{proof}

The previous corollary now allows us to extend the definition of f-invariant entropy to actions of finitely generated virtually free groups.

\begin{defn}
Let $\Gamma$ be a finitely generated virtually free group, and let $\Gamma$ act on a probability space $(X, \mu)$. If there is a generating partition for this action with finite Shannon entropy, then we define the \emph{f-invariant entropy of $\Gamma \acts (X, \mu)$} to be
$$f_\Gamma(X, \mu) = \frac{1}{|\Gamma : G|} \cdot f_G(X, \mu),$$
where $G$ is any free subgroup of $\Gamma$ of finite index, and the action of $G$ on $X$ is the restriction of the $\Gamma$ action. If there is no generating partition for this action with finite Shannon entropy, then the f-invariant entropy of $\Gamma \acts (X, \mu)$ is undefined.
\end{defn}

We point out that since f-invariant entropy is a measure conjugacy invariant for actions of finitely generated free groups, it is also a measure conjugacy invariant for actions of finitely generated virtually free groups. Also notice that by the previous corollary, $f_\Gamma(X, \mu)$ does not depend on the free subgroup of finite index chosen.

If $(K^\Gamma, \mu^\Gamma)$ is a Bernoulli shift over a finitely generated virtually free group $\Gamma$, then the f-invariant entropy of $\Gamma \acts (K^\Gamma, \mu^\Gamma)$ is
$$\sum_{k \in K} - \mu(k) \cdot \log(\mu(k))$$
provided that the support of $\mu$ is countable and the above sum is finite (this follows easily from the validity of this formula when $\Gamma$ is in fact free, as discussed by Bowen in \cite{B10a}). If the support of $\mu$ is not countable or the sum above is not finite, then the f-invariant entropy of this action is undefined (since in this case there is no generating partition having finite Shannon entropy, as proved by Kerr--Li in \cite{KL11b}). Moreover, it follows from \cite{B10b} and \cite{KL11b} that if $(K^\Gamma, \mu^\Gamma)$ and $(M^\Gamma, \lambda^\Gamma)$ are two Bernoulli shifts over a finitely generated virtually free group $\Gamma$, then they are measurably conjugate if and only if
$$\sH(\mu) = \sH(\lambda),$$
where $\sH(\mu)$ is defined to be $\sum_{k \in K} - \mu(k) \cdot \log(\mu(k))$ if the support of $\mu$ is countable, and is defined to be $\infty$ otherwise, and $\sH(\lambda)$ is defined similarly. So it immediately follows that for Bernoulli shifts over finitely generated virtually free groups for which f-invariant entropy is defined, the f-invariant entropy is a complete invariant for measure conjugacy. We also mention that many properties of the original f-invariant entropy immediately carry over to this new f-invariant entropy, such as the Abramov--Rohlin formula and (under a few assumptions) Juzvinskii's addition formula (see \cite{B10d} and \cite{BG}).

\begin{prob}
Let $G$ be a locally compact group and let $m$ be a Haar measure on $G$. Suppose that $\Gamma_1$ and $\Gamma_2$ are finitely generated free groups and are lattices in $G$. Let $G$ act measure preservingly on a standard probability space $(X, \mu)$. Is $f_{\Gamma_1}(X, \mu)$ defined if and only if $f_{\Gamma_2}(X, \mu)$ is defined? Are these f-invariant entropies related by their covolumes:
$$\frac{1}{m(\Gamma_1 \backslash G)} \cdot f_{\Gamma_1}(X, \mu) = \frac{1}{m(\Gamma_2 \backslash G)} \cdot f_{\Gamma_2}(X, \mu)?$$
\end{prob}

The above questions may only have positive answers under additional assumptions on $G$, such as $G$ being unimodular or a Lie group. An affirmative answer to these questions would allow f-invariant entropy to be extended to actions of locally compact groups which contain a free group lattice.

Now we turn to defining a numerical invariant for virtual measure conjugacy of actions of finitely generated virtually free groups. Recall that two probability measure preserving actions $G \acts (X, \mu)$ and $H \acts (Y, \nu)$ are \emph{virtually measurably conjugate} if there are subgroups of finite index $G' \leq G$ and $H' \leq H$ such that the restricted actions $G' \acts (X, \mu)$ and $H' \acts (Y, \nu)$ are measurably conjugate, meaning that there is a group isomorphism $\psi: G' \rightarrow H'$ and a measure space isomorphism $\phi: X \rightarrow Y$ such that $\phi(g' \cdot x) = \psi(g') \cdot \phi(x)$ for every $g' \in G'$ and $\mu$-almost every $x \in X$.

\begin{cor} \label{COR VFINV}
For $i = 1, 2$, let $\Gamma_i$ be a finitely generated virtually free group acting on a probability space $(X_i, \mu_i)$. Assume that for each $i$ there is a finite Shannon entropy generating partition for $\Gamma_i \acts (X_i, \mu_i)$. If $\Gamma_1 \acts (X_1, \mu_1)$ is virtually measurably conjugate to $\Gamma_2 \acts (X_2, \mu_2)$, then for any free subgroups of finite index $G_1 \leq \Gamma_1$ and $G_2 \leq \Gamma_2$, we have
$$\frac{1}{r(G_1) - 1} \cdot f_{G_1}(X_1, \mu_1) = \frac{1}{r(G_2) - 1} \cdot f_{G_2}(X_2, \mu_2),$$
where $r(G_i)$ is the rank of $G_i$.
\end{cor}

\begin{proof}
Since the actions are virtually measurably conjugate, there exist subgroups of finite index $H_1 \leq \Gamma_1$ and $H_2 \leq \Gamma_2$, a group isomorphism $\psi: H_1 \rightarrow H_2$, and a measure space isomorphism $\phi: (X_1, \mu_1) \rightarrow (X_2, \mu_2)$ with $\phi(h \cdot x) = \psi(h) \cdot \phi(x)$ for every $h \in H_1$ and $\mu_1$-almost every $x \in X_1$. Since $H_i$ is a finite index subgroup of $\Gamma_i$, $H_i$ is finitely generated and virtually free (\cite[Corollary IV.B.24]{Ha00}). Let $K_1$ be a free subgroup of $H_1$ of finite index, and set $K_2 = \psi(K_1)$. Then $K_2$ is a free subgroup of $H_2$ of finite index and $K_1 \acts (X_1, \mu_1)$ is measurably conjugate to $K_2 \acts (X_2, \mu_2)$. Now $K_i, G_i \leq \Gamma_i$ each have finite index, so $N_i = G_i \cap K_i$ is of finite index in $\Gamma_i$ as well. By Theorem \ref{THM MAIN} and Lemma \ref{LEM RANK} we have
$$\frac{1}{r(G_1) - 1} \cdot f_{G_1}(X_1, \mu_1) = \frac{1}{(r(G_1) - 1)|G_1 : N_1|} \cdot f_{N_1}(X_1, \mu_1)$$
$$= \frac{1}{r(N_1) - 1} \cdot f_{N_1}(X_1, \mu_1) = \frac{|K_1 : N_1|}{r(N_1) - 1} \cdot f_{K_1}(X_1, \mu_1)$$
$$= \frac{1}{r(K_1) - 1} \cdot f_{K_1}(X_1, \mu_1) = \frac{1}{r(K_2) - 1} \cdot f_{K_2}(X_2, \mu_2)$$
$$= \frac{|K_2 : N_2|}{r(N_2) - 1} \cdot f_{K_2}(X_2, \mu_2) = \frac{1}{r(N_2) - 1} \cdot f_{N_2}(X_2, \mu_2)$$
$$= \frac{1}{(r(G_2) - 1)|G_2 : N_2|} \cdot f_{N_2}(X_2, \mu_2) = \frac{1}{r(G_2) - 1} \cdot f_{G_2}(X_2, \mu_2).$$
\end{proof}

The previous corollary allows us to define a numerical invariant for virtual measure conjugacy among actions of finitely generated virtually free groups.

\begin{defn}
Let $\Gamma$ be a finitely generated virtually free group acting measure preservingly on a standard probability space $(X, \mu)$. If there is a generating partition having finite Shannon entropy, then the \emph{virtual f-invariant entropy of $\Gamma \acts (X, \mu)$} is defined as
$$\vf_\Gamma(X, \mu) = \frac{1}{r(G) - 1} \cdot f_G(X, \mu),$$
where $G$ is any free subgroup of finite index, $r(G)$ is the rank of $G$, and $G$ acts on $(X, \mu)$ by restricting the $\Gamma$ action. If there is no generating partition with finite Shannon entropy, then the virtual f-invariant entropy of this action is undefined.
\end{defn}

The corollary above shows that $\vf_\Gamma(X, \mu)$ does not depend on the free subgroup of finite index chosen (use $\Gamma_1 = \Gamma_2$ in that corollary) and is an invariant for virtual measure conjugacy.

We remark that $\vf_\Gamma(X, \mu)$ can be computed from $f_\Gamma(X, \mu)$ without choosing a free subgroup of finite index. In \cite{KPS73}, Karrass--Pietrowski--Solitar prove that any finitely generated virtually free group $\Gamma$ can be represented as the fundamental group of a finite graph of groups in which all vertex groups are finite. Furthermore, they showed that if $G \leq \Gamma$ is a free subgroup of finite index then the rank of $G$, $r(G)$, is given by
$$r(G) = 1 + |\Gamma : G| \cdot \left(\frac{1}{e_1} + \frac{1}{e_2} + \cdots + \frac{1}{e_k} - \frac{1}{v_1} - \frac{1}{v_2} - \cdots - \frac{1}{v_n} \right),$$
where $e_1, \ldots, e_k$, and $v_1, \ldots, v_n$ are the orders of the edge groups and vertex groups, respectively, corresponding to the representation of $\Gamma$ as the fundamental group of a finite graph of finite groups. Therefore
$$\vf_\Gamma(X, \mu) = \frac{1}{r(G) - 1} \cdot f_G(X, \mu) = \frac{|\Gamma : G|}{r(G) - 1} \cdot f_\Gamma(X, \mu)$$
$$= \left(\frac{1}{e_1} + \cdots + \frac{1}{e_k} - \frac{1}{v_1} - \cdots - \frac{1}{v_n} \right)^{-1} \cdot f_\Gamma(X, \mu).$$
However, it is unclear if there is a formula for $f_\Gamma(X, \mu)$ which avoids choosing a free subgroup of finite index.

\begin{prob}
For finitely generated virtually free groups $\Gamma$, find a formula for $f_\Gamma(X, \mu)$ which avoids choosing a free subgroup of finite index.
\end{prob}

We point out that for amenable groups $H \leq G$ and Kolmogorov--Sinai entropy, it is true that $h_H(X, \mu) = |G : H| \cdot h_G(X, \mu)$ (here $h_H$ and $h_G$ are the Kolmogorov--Sinai entropies of the $H$ and $G$ actions), however this fact does not allow one to define a numerical invariant for virtual measure conjugacy among actions of amenable groups. In proving Corollary \ref{COR VFINV} we relied on a property which is possibly unique to finitely generated virtually free groups. The property we used is that if $\Gamma$ is finitely generated and virtually free, and $G$ and $H$ are free subgroups of $\Gamma$ of finite index, then $\frac{|\Gamma: G|}{|\Gamma : H|} = \frac{r(G) - 1}{r(H) - 1}$. So the ratio of the indices of $G$ and $H$ in $\Gamma$ can be determined from the internal structure of $G$ and $H$ alone; no knowledge of $\Gamma$ is required.

We now show that $\vf_\Gamma(X, \mu)$ is a complete invariant for virtual measure conjugacy among the Bernoulli shifts on which it is defined.

\begin{prop}
For $i = 1, 2$, let $(K_i^{\Gamma_i}, \mu_i^{\Gamma_i})$ be a Bernoulli shift over a finitely generated virtually free group $\Gamma_i$. If the virtual f-invariant entropy $\vf_{\Gamma_i}(K_i^{\Gamma_i}, \mu_i^{\Gamma_i})$ is defined for each $i$, then $(K_1^{\Gamma_1}, \mu_1^{\Gamma_1})$ is virtually measurably conjugate to $(K_2^{\Gamma_2}, \mu_2^{\Gamma_2})$ if and only if $\vf_{\Gamma_1}(K_1^{\Gamma_1}, \mu_1^{\Gamma_1}) = \vf_{\Gamma_2}(K_2^{\Gamma_2}, \mu_2^{\Gamma_2})$.
\end{prop}

\begin{proof}
By Corollary \ref{COR VFINV}, it is necessary that the virtual f-invariant entropies of these actions agree. So suppose that they have the same virtual f-invariant entropy. We must show that the actions are virtually measurably conjugate.

For each $i$, pick a free subgroup $G_i \leq \Gamma_i$ of finite index. Let $H_1$ be a subgroup of $G_1$ with $|G_1 : H_1| = r(G_2) - 1$, and let $H_2$ be a subgroup of $G_2$ with $|G_2 : H_2| = r(G_1) - 1$. Such subgroups exist since $G_1$ and $G_2$ are finitely generated free groups. Then by Lemma \ref{LEM RANK} we have
$$r(H_1) - 1 = |G_1 : H_1| (r(G_1) - 1) = (r(G_2) - 1)(r(G_1) - 1)$$
$$= (r(G_2) - 1)|G_2 : H_2| = r(H_2) - 1.$$
Thus $H_1$ is group isomorphic to $H_2$.

If $\Delta_i$ is a transversal of the right $H_i$-cosets in $\Gamma_i$, then $H_i \acts (K_i^{\Gamma_i}, \mu_i^{\Gamma_i})$ is measurably conjugate to the Bernoulli shift $H_i \acts ((K_i^{\Delta_i})^{H_i}, (\mu_i^{\Delta_i})^{H_i})$. So we have
$$f_{H_1} \left((K_1^{\Delta_1})^{H_1}, (\mu_1^{\Delta_1})^{H_1} \right) = f_{H_1} \left(K_1^{\Gamma_1}, \mu_1^{\Gamma_1} \right) = \left(r(H_1) - 1 \right) \cdot \vf_{\Gamma_1} \left(K_1^{\Gamma_1}, \mu_1^{\Gamma_1} \right)$$
$$= \left(r(H_2) - 1 \right) \cdot \vf_{\Gamma_2} \left(K_2^{\Gamma_2}, \mu_2^{\Gamma_2} \right) = f_{H_2} \left(K_2^{\Gamma_2}, \mu_2^{\Gamma_2} \right) = f_{H_2} \left((K_2^{\Delta_2})^{H_2}, (\mu_2^{\Delta_2})^{H_2} \right).$$
Since $H_1$ is group isomorphic to $H_2$ and f-invariant entropy is a complete invariant for measure conjugacy among the Bernoulli shifts on which it is defined, we have (below $\cong$ denotes the measure conjugacy equivalence relation)
$$H_1 \acts (K_1^{\Gamma_1}, \mu_1^{\Gamma_1}) \cong H_1 \acts ((K_1^{\Delta_1})^{H_1}, (\mu_1^{\Delta_1})^{H_1})$$
$$\cong H_2 \acts ((K_2^{\Delta_2})^{H_2}, (\mu_2^{\Delta_2})^{H_2}) \cong H_2 \acts (K_2^{\Gamma_2}, \mu_2^{\Gamma_2}).$$
Thus the actions of $\Gamma_1$ and $\Gamma_2$ are virtually measurably conjugate as claimed.
\end{proof}

\end{document}